\documentclass[12pt]{article}

\usepackage{amsmath,amssymb,amsfonts,mathtools} 
\usepackage{xcolor,comment}
\usepackage[all]{xy}
\usepackage{amscd}
\usepackage{enumitem}

\usepackage[utf8]{inputenc}
\usepackage{authblk} 

\usepackage[margin=3cm]{geometry} 

\usepackage{amsthm} 

\theoremstyle{plain}
\newtheorem{Thm}{Theorem}[section]
\newtheorem{Prop}[Thm]{Proposition}
\newtheorem{Lem}[Thm]{Lemma}
\newtheorem{Cor}[Thm]{Corollary}

\theoremstyle{definition}
\newtheorem{Def}[Thm]{Definition}

\theoremstyle{remark}
\newtheorem{Rem}[Thm]{Remark}

\DeclareMathOperator{\rank}{rank}

\DeclareMathOperator{\coker}{coker}

\DeclareMathOperator{\Tr}{Tr}
\DeclareMathOperator{\Hom}{Hom}

\DeclareMathOperator{\End}{End}

\DeclareMathOperator{\Spec}{Spec}
\DeclareMathOperator{\length}{length}

\numberwithin{equation}{section}

\newcommand{\blfootnote}[1]{%
  \begingroup
  \renewcommand{\thefootnote}{}
  \footnotetext{#1}
  \endgroup
}

\title{Moduli space of irregular singular $\lambda$-connections
and the relative symplectic form}

\author[1]{Michi-aki Inaba\thanks{E-mail: inaba@cc.nara-wu.ac.jp}}
\author[2]{Arata Komyo\thanks{E-mail: akomyo@sci.u-hyogo.ac.jp}}

\affil[1]{Faculty of Science, Nara Women's University}
\affil[2]{Graduate School of Science, University of Hyogo}

\date{}

\begin{document}

\maketitle


\begin{abstract}
We study the moduli space of $\lambda$-connections 
with ramified singularities on a smooth projective curve $X$. 
We construct a quasi-projective relative moduli scheme $M^{\alpha}_{\mathrm{Hod}} \to \Lambda$ 
of $\alpha$-stable $\lambda$-connections over a base parameter space $\Lambda$. 
This moduli space connects the irregular de Rham moduli space ($\lambda=1$) 
and the irregular Dolbeault moduli space of Higgs bundles equipped with non-zero irregular local exponents ($\lambda=0$).
Furthermore, we show that $M^{\alpha}_{\mathrm{Hod}}$ is smooth over $\Lambda$ 
and carries a canonical relative symplectic form $\omega_{M^{\alpha}_{\mathrm{Hod}}/\Lambda}$, 
yielding a natural Poisson structure on the total space.
\end{abstract}

\section{Introduction}

\subsection{Background}

In non-abelian Hodge theory, Simpson introduced a comprehensive framework incorporating three fundamental moduli spaces: the \textit{de Rham} (connections), \textit{Dolbeault} (Higgs bundles), and \textit{Betti} (character varieties) moduli spaces, endowing each with the structure of an algebraic variety \cite{Simp-I, Simp-II, Simp-III}. 
Originally motivated by the construction of Hitchin's twistor space \cite{Hitchin92}
and following a suggestion by Deligne,
Simpson introduced the notion of $\lambda$-connections in \cite{Simp-IV}. 
Let $X$ be a smooth complex projective curve, and let $D$ be an effective divisor on $X$. 
For a vector bundle $E$ on $X$ and a complex parameter $\lambda \in \mathbb{C}$, a $\mathbb{C}$-linear homomorphism 
$
  \nabla \colon E \rightarrow E \otimes \Omega^1_X(D)
$
is called a \textit{$\lambda$-connection} if it satisfies
\[
  \nabla(fa) = \lambda a \otimes \mathrm{d}f  + f\nabla(a)
\]
for all $f \in \mathcal{O}_X$ and $a \in E$. 
When $\lambda \neq 0$, the scaled map $\lambda^{-1}\nabla$ is a connection,
whereas for $\lambda=0$, a $0$-connection is a Higgs field. 
Thus, $\lambda$-connections provide a unified framework to handle connections and Higgs bundles simultaneously. Consequently, the moduli space of $\lambda$-connections, known as the \textit{Hodge moduli space}, provides an algebraic deformation connecting the de Rham and Dolbeault moduli spaces. 
By virtue of Simpson's work \cite{Simp-IV}, a quotient by the canonical $\mathbb{C}^*$-action on the Hodge moduli space induces an algebraic compactification of the de Rham moduli space, which is useful for studying the algebraic structures
of both the de Rham and Dolbeault moduli spaces.

The above framework of holomorphic $\lambda$-connections 
can be extended to meromorphic cases. 
When $D$ is reduced, 
Simpson established the fundamental correspondence 
among the three objects (de Rham, Dolbeault, and Betti) in \cite{Simp-0}. 
The algebraic construction of logarithmic Dolbeault moduli spaces was established in \cite{Nits-91, Simp-II, Simp-III}, 
and its parabolic counterpart was obtained in \cite{yoko-1} in a broader setting that accommodates non-reduced divisors. 
Similarly, the algebraic construction of logarithmic de Rham moduli spaces was established in \cite{Nits-93, Simp-II, Simp-III}, 
while the parabolic version was achieved in \cite{Inaba-Iwasaki-Saito1, Inaba1}, 
where the Riemann--Hilbert map was also investigated, leading to applications in the study of isomonodromic deformations. 
From a differential-geometric viewpoint, the logarithmic de Rham and Dolbeault moduli spaces were constructed in \cite{Konno-93, Nakaj-94}. 
The algebraic Betti moduli space was constructed in \cite{Yam}, refining the Riemann--Hilbert correspondence in \cite{Inaba-Iwasaki-Saito1, Inaba1}. 
The moduli space of logarithmic $\lambda$-connections (Hodge moduli space) was constructed algebraically in \cite{Alfaya-Gomez}, with applications to a Torelli-type theorem for the Deligne--Hitchin moduli space. 
In \cite{BIKS}, the method of $\lambda$-connections was also utilized to investigate the algebraic structure of the logarithmic de Rham moduli space.

When the polar divisor $D$ is nonreduced,
the Hukuhara--Levelt--Turrittin normal form provides a classification of formal connections
over $\mathbb{C}(\!(w)\!)$, where $w^l = z$ for a local coordinate $z$ at $x \in D$.
The singularity $x$ is called \emph{unramified} if $l = 1$
and \emph{ramified} if $l > 1$. 
An important principle for formulating the moduli problem is that 
formal types should determine the classification over the formal power series ring 
rather than over the field of formal Laurent series.
Indeed, there is a substantial gap between the classifications
over the formal power series ring and over the field of formal Laurent series.
In the generic unramified case, this difference lies
only in the integral parts of the residues.
A precise construction of the moduli space requires
replacing the formal type of a connection
with the structure of its restriction to the polar divisor.
In the generic unramified case, the formal type can be uniquely recovered
from its restriction to the polar divisor.
Based on this crucial consequence, the algebraic de Rham moduli space 
and its symplectic structure were constructed in \cite{Inaba-Saito-13}
in the generic unramified case. 
Through a careful argument, a similar consequence holds in the generic ramified case, 
based on which the algebraic de Rham moduli space 
and its symplectic structure were constructed in \cite{Inaba3, Inaba4}.
For the corresponding Betti moduli spaces,
see \cite{Boalch-1, Boalch-2, Put-Saito}.
A differential-geometric construction of the Dolbeault and de Rham moduli spaces
was provided in \cite{Biquard-Boalch}, which established
the non-abelian Hodge correspondence in the generic unramified case.
See \cite{Mochi11} for a comprehensive generalization
of the non-abelian Hodge correspondence on higher-dimensional varieties.

In this paper, we construct the moduli space of $\lambda$-connections
allowing ramified singularities,
which connects the irregular de Rham moduli space
and its Dolbeault counterpart.

\subsection{Main results}

Let $X$ be a smooth projective algebraic curve of genus $g$ over $\mathbb{C}$. Take an effective divisor $D$ on $X$ and write
\[
  D = \sum_{x \in D} m_x x.
\]
At each point $x \in D$, choose a local generator $z$ of the maximal ideal of $\mathcal{O}_{X,x}$. 
Fix integers $r, d$ with $r \geq 2$.
For each $x\in D$, take a partition $r = r^x_1 + \cdots + r^x_{s_x}$
where $r^x_k > 0$ for all $1 \leq k \leq s_x$. 
If $m_x=1$,
we further assume that $r^x_1=\cdots=r^x_{s_x}=1$.
Set
\[
  \mathcal{A}^x_i := \mathcal{O}_{m_x x}[w] / (w^{r^x_i} - z)
\]
for $1 \leq i \leq s_x$.

The local invariants of $\lambda$-connections associated with each pole $x \in D$ are defined as follows:
\begin{Def} \label {def: generic ramified local exponent data intro}
Take polynomials
\[
a^{x,i}_0(T), a^{x,i}_1(T), \ldots, a^{x,i}_{r^x_i-1}(T) \in \mathbb{C}[T]
\]
such that $\deg a^{x,i}_0(T) \leq r^x_i-1$ and $\deg a^{x,i}_k(T) \leq r^x_i-2$ for $1 \leq k \leq r^x_i-1$.
Set 
\[
\begin{aligned} 
\nu^{x,i}_k(z) & := a^{x,i}_k(z)\frac{\mathrm{d}z}{z^{m_x}} 
\in \Omega^1_X(D)|_{m_x x} \quad (0 \leq k \leq r^x_i-1),  \\
\nu^{x,i}(w) & := \nu^{x,i}_0(z) + \nu^{x,i}_1(z)w + \cdots + \nu^{x,i}_{r^x_i-1}(z)w^{r^x_i-1}
\in \mathcal{A}^x_i \otimes \Omega^1_X(D)|_{m_x x}. 
\end{aligned}
\]
For a given $\lambda \in \mathbb{C}$,
we say that the data $\nu = (\nu^{x,i}(w))_{x \in D, \, 1 \leq i \leq s_x}$
is a \emph{generic ramified local exponent data} if the following conditions hold:
\begin{enumerate}
\item[\textup{(i)}] $a^{x,1}_0(0), \ldots, a^{x,s_x}_0(0)$ are mutually distinct for $x \in D$ with $m_x \geq 2$,
\item[\textup{(ii)}] $a^{x,i}_1(0) \neq 0$ holds for all $x \in D$ and $1 \le i \le s_x$ with $r^x_i \ge 2$, and
\item[\textup{(iii)}] $\lambda d + \sum_{x \in D} \sum_{i=1}^{s_x} r^x_i \operatorname{Res}_x(\nu^{x,i}_0(z)) = 0$ 
(the Fuchs relation).
\end{enumerate}
\end{Def}

Denote $\mathcal{A} := (\mathcal{A}^x_i)_{x \in D, \, 1 \leq i \leq s_x}$ and define
\[
  \Lambda := \left\{ (\lambda, \nu) \;\middle|\;
  \begin{array}{l}
    \lambda \in \mathbb{C}, \\ 
    \nu \text{ is a generic ramified local exponent data for $\lambda$}
  \end{array}
  \right\}.
\]
Let $E$ be an algebraic vector bundle on $X$ of rank $r$ and let $\nabla\colon E\rightarrow E\otimes\Omega^1_X(D)$ be a $\lambda$-connection.
For $(\lambda, \nu) \in \Lambda$,
we introduce the notion of a parabolic $(\nu, \mathcal{A})$-frame, which characterizes the formal structure of $(E,\nabla)$ at $D$.

\begin{Def} \label{def parabolic frame in Intro}
We say that $\{l^x_i, L^x_i\}$ is a \emph{parabolic $(\nu, \mathcal{A})$-frame} for a $\lambda$-connection $(E, \nabla)$ if:
\begin{enumerate}
  \item[\textup{(i)}] $E|_{m_x x} = l^x_1 \supset \cdots \supset l^x_{s_x} \supset l^x_{s_x+1} = 0$ is a filtration by $\mathcal{O}_{m_x x}$-submodules such that $l^x_i / l^x_{i+1} \cong \mathcal{O}_{m_x x}^{\oplus r^x_i}$ and $\nabla|_{m_x x}(l^x_i) \subset l^x_i \otimes \Omega^1_X(D)|_{m_x x}$ for each $1 \leq i \leq s_x$;
  \item[\textup{(ii)}] $L^x_i$ is an $\mathcal{A}^x_i$-submodule of $\operatorname{Hom}_{\mathcal{O}_{m_x x}}(\mathcal{A}^x_i, l^x_i / l^x_{i+1})$ satisfying:
  \begin{enumerate}
    \item[\textup{(ii-1)}] $L^x_i$ is a free $\mathcal{A}^x_i$-module of rank one,
    \item[\textup{(ii-2)}] a generator $\sigma^{x,i}$ of $L^x_i$ is an isomorphism $\sigma^{x,i} \colon \mathcal{A}^x_i \xrightarrow{\,\sim\,} l^x_i / l^x_{i+1}$ of $\mathcal{O}_{m_x x}$-modules, and
    \item[\textup{(ii-3)}] the equality
    \[
      ((\sigma^{x,i})^{-1} \otimes \mathrm{id}_{\Omega^1_X(D)|_{m_x x}}) \circ \overline{\nabla}^x_i \circ \sigma^{x,i} = \nu^{x,i}(w) \, \mathrm{id}_{\mathcal{A}^x_i}
    \]
    holds, where $\overline{\nabla}^x_i \colon l^x_i / l^x_{i+1} \longrightarrow l^x_i / l^x_{i+1} \otimes \Omega^1_X(D)|_{m_x x}$ is the induced homomorphism.
  \end{enumerate}
We call this $\mathcal{A}^x_i$-submodule $L^x_i$ a \emph{$(\nu^{x,i}(w), \mathcal{A}^x_i)$-frame} 
(see Definition \ref{def: (nu,A)-frame}).
\end{enumerate}
\end{Def}

The formal data is uniquely reconstructed from a parabolic $(\nu, \mathcal{A})$-frame, 
as explained in Remark~\ref{rem: formal data of parabolic framing structure} below. 
In Definition~\ref{def: stability of lambda connection} below, 
we will introduce the notion of $\alpha$-stability for $\lambda$-connections 
with parabolic weights $\alpha$.

Our first main result is the following:

\begin{Thm} \label{thm: main_moduli_construction}
There exists a relative coarse moduli scheme
\[
\pi\colon
  M^{\alpha}_{\mathrm{Hod}} \longrightarrow \Lambda
\]
of $\alpha$-stable $\lambda$-connections $(E, \nabla)$ on $X$ with a parabolic frame. 
Furthermore, $M^{\alpha}_{\mathrm{Hod}}$ is a quasi-projective scheme over $\Lambda$. 
Its fiber over $(1, \nu)$ is the moduli space $M^{\alpha}_{\mathrm{dR}}(\nu)$ of $\alpha$-stable connections 
with a parabolic $(\nu, \mathcal{A})$-frame, 
and its fiber over $(0, \mu)$ is the moduli space $M^{\alpha}_{\mathrm{Higgs}}(\mu)$ 
of $\alpha$-stable Higgs bundles with a parabolic $(\mu, \mathcal{A})$-frame.
\end{Thm}

Precisely, we use the \emph{framing structure} (see Definition~\ref{def: framing structure}) 
to define $\alpha$-stability for the moduli space $M^{\alpha}_{\mathrm{Hod}}$ 
in Theorem~\ref{thm: main_moduli_construction}. 
In this paper, we mainly discuss the formulation of a framing structure for $\lambda$-connections, 
where the parabolic frame in Definition~\ref{def parabolic frame in Intro} appears as its normalization.

The notion of framing structure is introduced from the authors' primary motivation 
to seek a moduli-theoretic compactification 
of the de Rham moduli spaces corresponding to the Painlev\'e equations. 
In this specific context, $X = \mathbb{P}^1$, $r = 2$, and $\deg D = 4$. 
For reduced $D$, such a compactification was achieved in \cite{Inaba-Iwasaki-Saito2} 
by introducing the notion of a \emph{$\phi$-connection}. 
This compactification is isomorphic to the canonical smooth compactification of the space of initial conditions 
for the Painlev\'e VI equation, constructed by Okamoto \cite{Okamo-79}, 
and it provides a resolution of singularities of Simpson's compactification 
of the de Rham moduli space using $\lambda$-connections in \cite{Simp-IV}. 
For non-reduced $D$, we aim to establish a similar compactification in future work 
by introducing the notion of a \emph{framing structure for $\phi$-connections}; 
the resulting moduli spaces are expected to cover the rational surfaces in \cite{Saito-Takebe-Terajima, Sakai}, 
which completely determine the Painlev\'e equations through their geometry.

For rank~$2$ connections over $\mathbb{P}^1$ with ramified singularities, 
Diarra--Loray \cite{Diarra-Loray} explicitly describe
a Zariski open subset of the de Rham moduli space
by specifying the condition of formal type,
which is essentially equivalent to the notion of
{\it parabolic framing structure} (see Definition \ref{def: parabolic framing structure}). 
Over this Zariski open subset, Komyo \cite{Komyo} studies the symplectic structure 
and explicitly calculates the Hamiltonians for generalized isomonodromic deformations.

The existence of the moduli spaces appearing as the fibers of $\pi$ 
in Theorem~\ref{thm: main_moduli_construction} is already known: 
for $\lambda \neq 0$, they were constructed in \cite{Inaba3}, 
and for $\lambda = 0$, the existence follows from \cite{yoko-1}. 
A remarkable point in the formulation of Theorem~\ref{thm: main_moduli_construction} is 
that the local exponents in the fiber over $\lambda = 0$ need not be zero, 
in contrast to other literature (e.g., \cite{Alfaya-Gomez}). 
This formulation has an advantage, as it naturally fits into the non-abelian Hodge correspondence 
in a general setting (see \cite{Simp-V}).

Our second main theorem is stated as follows:

\begin{Thm} \label {Thm: main theorem}
The following statements hold:
\begin{enumerate}
\item[\textup{(1)}]
The moduli space $M^{\alpha}_{\mathrm{Hod}}$ of
$\alpha$-stable $\lambda$-connections with a parabolic frame
is smooth over $\Lambda$.
\item[\textup{(2)}]
The dimension of the fiber $M^{\alpha}_{\mathrm{Hod}}(\lambda,\nu)$
of $M^{\alpha}_{\mathrm{Hod}}$ over each point
$(\lambda,\nu)$ of $\Lambda$
is
\[
2r^2(g-1)+2+r(r-1) \deg (D).
\]

\item[\textup{(3)}]
There is a relative symplectic form
$\omega_{M^{\alpha}_{\mathrm{Hod}}/\Lambda}
\in
H^0(M^{\alpha}_{\mathrm{Hod}},
\Omega^2_{M^{\alpha}_{\mathrm{Hod}}/\Lambda})$
on the Hodge moduli space
$M^{\alpha}_{\mathrm{Hod}}$.
\end{enumerate}
\end{Thm}

It is worth mentioning that the relative symplectic form
in Theorem~\ref{Thm: main theorem} yields a canonical Poisson structure
on the total space $M_{\mathrm{Hod}}^{\alpha}$
(Corollary~\ref{Cor: Poisson structure}).
There are several studies on the symplectic or Poisson structure
on the moduli space of Higgs bundles
(\cite{Mark-94, Botta-95}).
The restriction of the Poisson structure on $M_{\mathrm{Hod}}^{\alpha}$ to each Hitchin fiber
coincides with or enlarges
the known Poisson structure on the moduli space of Higgs bundles.

By extending the base space $\Lambda$ to the moduli space of 
pairs $(X,D)$ consisting of an algebraic curve $X$ and a divisor $D$ 
equipped with the parameters in $\Lambda$, 
one can construct an isomonodromic family on the locus where $\lambda \neq 0$ in the Hodge moduli space. 
Furthermore, the relative symplectic form $\omega_{M^{\alpha}_{\mathrm{Hod}}/\Lambda}$ 
in Theorem~\ref{Thm: main theorem}~(3) can be extended 
to an isomonodromic closed $2$-form on this locus, as in \cite[Corollary 11.8]{Inaba4}. 

An important future direction would be to explore the behavior of these isomonodromic deformations 
as $\lambda$ tends to $0$. 
This inquiry is motivated by the desire to understand the deep connections among 
isomonodromic deformations, exact WKB analysis, and non-abelian Hodge theory
(e.g.\ \cite{GMN13AiM, DM14, CW19, Iwaki25}).

\subsection{The organization of this paper}

In Section~\ref{sect: Framing str}, 
we study framing structures for $\lambda$-connections. 
Specifically, in Section~\ref{subsect framing str}, 
we introduce the notion of framing structures.
In Section~\ref{Subsect relation with ram conn}, 
we recall ramified structures from \cite{Inaba3}
and establish a bijective correspondence 
between ramified structures and framing structures. 
In Section~\ref{subsect nu A frame}, 
we discuss the normalization of framing structures, 
which leads to the notion of a parabolic frame (Definition~\ref{def parabolic frame in Intro}).

Section~\ref{sect: Moduli space} is devoted to the algebraic construction of the moduli space of $\lambda$-connections 
with generic ramified singularities, 
thereby establishing Theorem~\ref{thm: main_moduli_construction}.

Section~\ref{sect: Defor lambda conn} is dedicated to the proof of Theorem~\ref{Thm: main theorem}. 
Statements~(1) and (2) of Theorem~\ref{Thm: main theorem} are proved in Section~\ref{subsect relative tangent sp}, 
and Statement~(3) is proved in Section~\ref{subsect rel symp str}. Following the strategy in \cite{Inaba2}, 
we construct a complex $\mathcal{F}^\bullet$ via a linear endomorphism factorization associated 
with a parabolic frame. 
We show that its hypercohomology $\mathbb{H}^1(\mathcal{F}^\bullet)$ 
is isomorphic to the tangent space of a fiber of $M^{\alpha}_{\mathrm{Hod}} \to \Lambda$, 
and carries a natural, nondegenerate bilinear pairing $\omega_{\mathbb{H}^1(\mathcal{F}^\bullet)}$. 
This pairing yields the smoothness, dimension count, and relative symplectic structure of the moduli space. 
In particular, the closedness of the 2-form in Statement~(3) is established in Proposition~\ref{Prop: d-closed} 
via parameter unfolding into unramified irregular singularities, reducing the problem to the unramified case in \cite{Inaba2}.

\subsection*{Acknowledgments}
The authors were supported by JSPS KAKENHI Grant Numbers 22H00094, 
and the second author by Grant Number 24K06674.

\subsection*{Disclosure of AI use}
The authors would like to thank Gemini (Google) for assistance with English proofreading and stylistic improvements. 
The mathematical content, proofs, and conclusions in this paper are entirely the work of the authors, who take full responsibility for the final text.

\section{Framing structure} \label{sect: Framing str}

In this section, we introduce and study the notion of a framing structure
for a $\lambda$-connection.
It is equivalent to the notion of
a $(\nu,{\mathcal A})$-frame, forming the essential part of
Definition~\ref{def parabolic frame in Intro}.
For a connection, this notion is equivalent to the notion of a ramified structure given in \cite{Inaba3}.

\subsection{$(\nu,\boldsymbol{\delta})$-framing structure}
\label {subsect framing str}

Let $X$ be a smooth projective complex algebraic curve.
For $x\in X$, take a generator $z$ of the 
maximal ideal $\mathfrak{m}_x$ of the local ring ${\mathcal O}_{X,x}$
at $x$.
Then we have 
$\widehat{\mathcal O}_{X,x}\cong\mathbb{C}[[z]]$.
For an integer $m$ with $m\geq 2$, set
$\widehat{\Omega}(mx):=
\Omega^1_X(mx)\otimes\widehat{\mathcal O}_{X,x}$,
where $\Omega^1_X(mx)$ is the sheaf of rational differential
forms on $X$ which admits poles at $x$ of order at most $m$.
Take $\lambda\in\mathbb{C}$.
Let $\widehat{E}$ be a free $\widehat{\mathcal O}_{X,x}$-module
of rank $r$ on $X$ and let
$\widehat{\nabla}\colon \widehat{E}\rightarrow 
\widehat{E}\otimes\widehat{\Omega}(mx)$
be a formal $\lambda$-connection,
that is, a $\mathbb{C}$-linear map
satisfying $\widehat{\nabla}(fv)=\lambda v\otimes  \mathrm{d} f+f\widehat{\nabla}(v)$
for $f\in\widehat{\mathcal O}_{X,x}$ and $v\in\widehat{E}$.

Let $r$ be a positive integer. 
Set
\[
\mathcal{A} := \mathcal{O}_{m x}[w] / (w^{r} - z).
\]
Denote $\widehat{\Omega}(mx)|_{mx}$ by
$\overline{\Omega}(mx)$,
where  $mx=\Spec{\mathcal O}_{X,x}/\mathfrak{m}_x^m$.
Take polynomials
\[
\begin{cases}
a_0(T)=
a_{0,0}+a_{0,1}T+\cdots+a_{0,m-1}T^{m-1}
& (a_{0,i}\in\mathbb{C})
\\
a_k(T)=
a_{k,0}+a_{k,1}T+\cdots+a_{k,m-2}T^{m-2}
& (a_{k,i}\in\mathbb{C},\: 0\leq k\leq r-1)
\end{cases}
\]
in $T$ such that $a_1(0)=a_{1,0}\neq 0$.
Set
\begin{align*}
a(w)
&:=a_0(z)+a_1(z)w+\cdots+a_{r-1}(z)w^{r-1} \in \mathcal{A},  \\
\nu_k(z)
&:=a_k(z)\frac{\mathrm{d}z}{z^m}
\qquad (0\leq k\leq r-1), \ \ \text{and}
\\
\nu(w)
&:=
\nu_0(z)+\nu_1(z)w+\cdots+\nu_{r-1}(z)w^{r-1}=
a(w)\frac{\mathrm{d} z}{z^m} \in \mathcal{A} \otimes \overline{\Omega}(mx).
\end{align*}

We characterize a formal type of the
$\lambda$-connection $(\widehat{E},\widehat{\nabla})$
by considering its restriction $ (\widehat{E}|_{mx},\widehat{\nabla}|_{mx})$.
Here $\widehat{E}|_{mx}=\widehat{E}\otimes{\mathcal O}_{mx}
=\widehat{E}\otimes\mathbb{C}[z]/(z^m)$,
which is a free ${\mathcal O}_{mx}$-module of rank $r$,
and 
$\widehat{\nabla}|_{mx} \colon \overline{E}\rightarrow
\overline{E}\otimes\overline{\Omega}(mx)$
is the ${\mathcal O}_{mx}$-homomorphism
canonically induced  from $\widehat{\nabla}$.
To formulate this notion in a unified framework independent of $\lambda$ 
(whether it arises from a connection or a Higgs bundle), 
we treat this pair purely as an algebraic local data.
Take a tuple of complex numbers 
$\boldsymbol{\delta}=(\delta_0,\delta_1,\ldots,\delta_{r-1})$
satisfying 
\begin{equation}
\label {trace condition}
\delta_0+\delta_1+\cdots+\delta_{r-1}=0.
\end{equation}
Set
\[
\nu(w)^{(i)}_{\boldsymbol{\delta}}:=
\nu_0(z)+\nu_1(z)w+\cdots+\nu_{r-1}(z)w^{r-1}
+\delta_i\frac{\mathrm{d}z}{z},
\]
which induces a homomorphism 
$\mathcal{A} \rightarrow 
\mathcal{A} \otimes \overline{\Omega}(mx)$,
denoted by 
$\nu(w)_{\boldsymbol{\delta}}^{(i)}  \mathrm{id}_{\mathcal A}$.

\begin{Def} \label{def: framing structure}
Let $(\overline{E},\overline{\nabla})$ be a pair consisting of
a free ${\mathcal O}_{mx}$-module $\overline{E}$ of rank $r$
and a homomorphism
$\overline{\nabla}\colon\overline{E}\rightarrow\overline{E}\otimes\overline{\Omega}(mx)$
of $\mathcal{O}_{mx}$-modules.
Regard $\Hom_{{\mathcal O}_{mx}} ({\mathcal A}, \, \overline{E})$
as an ${\mathcal A}$-module via
$f(w)\cdot h=h\circ f(w)\mathrm{id}_{\mathcal A}$
for $f(w)\in{\mathcal A}$ and 
$h\in\Hom_{{\mathcal O}_{mx}} ({\mathcal A}, \, \overline{E})$.
For $i=0,1,\ldots,r-1$, define a homomorphism
\[
{\mathcal D}_{\overline{\nabla},\boldsymbol{\delta}}^{(i)}
\colon
\Hom_{{\mathcal O}_{mx}}
({\mathcal A},\overline{E})
\longrightarrow
\Hom_{{\mathcal O}_{mx}} ({\mathcal A},\overline{E}\otimes\overline{\Omega}(mx))
\]
by
\[
 {\mathcal D}_{\overline{\nabla},\boldsymbol{\delta}}^{(i)}(h)=
\overline{\nabla}\circ h-(h\otimes \mathrm{id}_{\overline{\Omega}(mx)})\circ
\nu(w)_{\boldsymbol{\delta}}^{(i)}  \mathrm{id}_{\mathcal A}
\]
for $h\in\Hom_{{\mathcal O}_{mx}}({\mathcal A},\overline{E})$.
A \emph{$(\nu(w),\boldsymbol{\delta})$-framing structure
$(\{N_i\},\{\overline{l}_i\})$ on $(\overline{E},\overline{\nabla})$}
consists of a filtration
\begin{equation}
\label {filtration on E|_{mx}}
 z^{m-1}\overline{E}=\overline{l}_0\supset \overline{l}_1
 \supset\cdots\supset \overline{l}_{r-1}\supset \overline{l}_r=0
\end{equation}
by ${\mathcal O}_{mx}$-submodules
and a filtration
\begin{equation}
\label {filtration on Hom space}
\Hom_{{\mathcal O}_{mx}} ({\mathcal A}, \, \overline{E})
\supset N_0\supset \cdots\supset N_{rm-1}\supset N_{rm}=0
\end{equation}
by ${\mathcal A}$-submodules
such that  the following hold:
\begin{itemize}
\item[{\rm (i)}] $\length \overline{l}_i=r-i$ for $0\leq i\leq r$,
$\length N_i=rm+r(r-1)/2-i$
for $0\leq i\leq rm-r$
and
$\length N_{rm-r+j}=r(r+1)/2-j(j+1)/2$
for $0\leq j\leq r$;
\item[{\rm (ii)}]
$wN_i\subset N_{i+1}$ for $0\leq i\leq rm-1$
and each $h\in N_{rm-r}$ satisfies
$h(w^k)\in \overline{l}_{k}$ for $0\leq k\leq r-1$;
\item[{\rm (iii)}]
${\mathcal D}_{\overline{\nabla},\boldsymbol{\delta}}^{(i)}(w^iN_0)\subset 
\Hom_{{\mathcal O}_{mx}}({\mathcal A},\overline{l}_{i+1})\otimes\overline{\Omega}(mx)$ 
holds for $ 0 \leq i \leq r-1$; and
\item[{\rm (iv)}]
$w^{rm-1}N_0\neq 0$.
\end{itemize}
\end{Def}

For the homomorphism $\overline{\nabla}$ 
in Definition~\ref{def: framing structure}, 
define a homomorphism $R\colon \overline{E}\rightarrow \overline{E}$ of $\mathcal{O}_{mx}$-modules 
by the equation
\begin{equation}
\label {restriction of traceless connection}
 R \otimes \frac{\mathrm{d}z}{z^m} = \overline{\nabla} - \nu_0(z) \cdot \mathrm{id}_{\overline{E}}.
\end{equation}

\begin{Lem} \label {Lem: alternative condition of FS}
Let $(\{N_i\},\{\overline{l}_i\})$ be a pair of filtrations
of the forms \eqref{filtration on E|_{mx}} and \eqref{filtration on Hom space}
in Definition \ref{def: framing structure}
which satisfies the conditions \textup{(i)}, \textup{(ii)}, \textup{(iii)} of Definition \ref{def: framing structure}
and the condition
$w^{2r-2}N_0\neq 0$
instead of \textup{(iv)}.
Then $(R|_{z^{m-1}\overline{E}})^r=0$ and
$\dim(\ker(R|_{z^{m-1}\overline{E}}))=1$ hold.
\end{Lem}

\begin{proof}
Take $h'\in N_0$ such that $w^{2r-2}h'\neq 0$
and $w^{2r-1}h'=0$.
Let $E'$ be the image of the composition
\[
{\mathcal A} w^{r-1}h'\otimes {\mathcal A}
\hookrightarrow
\Hom_{{\mathcal O}_{mx}}({\mathcal A},\overline{E})\otimes{\mathcal A}
\xrightarrow{\;\pi\;} \overline{E},
\]
where $\pi$ is defined by
$\pi(h\otimes a)=h(a)$
for $h\in\Hom_{{\mathcal O}_{mx}}({\mathcal A},\overline{E})$
and $a\in{\mathcal A}$.
Since $zw^{r-1}h'=w^{2r-1}h'=0$,
we have $zE'=0$, which means $E'\subset z^{m-1}\overline{E}$.
By Definition \ref {def: framing structure} (iii), we have
\[
0={\mathcal D}^{(r-1)}_{\overline{\nabla},\boldsymbol{\delta}}(w^{r-1}h')
=\overline{\nabla}\circ (w^{r-1}h')
-(w^{r-1}h')\circ\nu(w)\mathrm{id}_{\mathcal A},
\]
which is equivalent to
$
R\circ(w^{r-1}h')=w^{r-1}h'\circ(a(w)-a_0(z))\mathrm{id}_{\mathcal A}
$,
where $R$ is defined by \eqref {restriction of traceless connection}.
This implies
\[
R^j\circ(w^{r-1}h')=w^{r-1}h'\circ 
(a(w)-a_0(z))^j\mathrm{id}_{\mathcal A}
\]
for $j\geq 0$.
Recall that $z(w^{r-1}h')=w^{2r-1}h'=0$.
So we have
\[
R^r\circ (w^{r-1}h')=
w^{r-1}h'\circ (a(w)-a_0(z))^r\mathrm{id}_{\mathcal A}
=w^{r-1}h'\circ a_1(z)^rw^r\mathrm{id}_{\mathcal A}=0
\]
and
\[
R^{r-1}\circ (w^{r-1}h')
=
w^{r-1}h'\circ (a(w)-a_0(z))^{r-1}\mathrm{id}_{\mathcal A}
=
a_1(z)^{r-1}w^{2r-2}h'\neq 0
\]
because $w^{2r-2}h'\neq 0$ and $a_1(z)\in{\mathcal O}_{mx}^{\times}$.
From the definition of $E'$, we have
$R^{r-1}|_{E'}\neq 0$ and $R^r|_{E'}=0$.
So the minimal polynomial of $R|_{E'}$ becomes $T^r$.
Since the minimal polynomial divides the characteristic polynomial of $R|_{E'}$
whose degree is $\dim E'$,
we have
$
\dim E'\geq \deg T^r=r=\dim z^{m-1}\overline{E}
$.
Thus we have $E'=z^{m-1}\overline{E}$.
Since the minimal polynomial of $R|_{z^{m-1}\overline{E}}$ is $T^r$,
its eigenspace is one dimensional.
\end{proof}

\begin{Lem} \label {Lem: structure of hom space}
Let $(\{N_i\},\{\overline{l}_i\})$ satisfy the conditions~\textup{(i)}, \textup{(ii)}, and \textup{(iii)}
of Definition~\ref{def: framing structure} and the condition~\textup{(iv')} below instead of \textup{(iv)}:
\begin{itemize}
\item[\textup{(iv')}] $w^{2r-2}N_0\neq 0$ and 
$\ker\big(N_0\xrightarrow{w^{r-1}}N_0\big)=N_{rm-r+1}$.
\end{itemize}
Then we have
\[
N_0 \; \cong \;
\mathbb{C}[w]/(w^{rm})\oplus
\mathbb{C}[w]/(w^{r-1})\oplus\cdots\oplus\mathbb{C}[w]/(w)
\]
as an ${\mathcal A}$-module.
In particular, the condition \textup{(iv)} of Definition \ref{def: framing structure} holds.
\end{Lem}

\begin{proof}
Since $\overline{E}\cong{\mathcal O}_{mx}^{\oplus r}$,
we may regard
$
N_0  \subset
\Hom_{{\mathcal O}_{mx}}({\mathcal A},\overline{E})
\cong{\mathcal A}^{\oplus r}
$.
So we can write
\begin{equation} \label {candidate of structure of N_0}
N_0 \; \cong \;
\mathbb{C}[w]/(w^{s_1})\oplus\cdots\oplus\mathbb{C}[w]/(w^{s_r}),
\end{equation}
where $s_1\geq s_2\geq\cdots\geq s_r\geq 0$.

We can see $\dim \left( \ker\left(w|_{N_0}\right)\right)\leq r$
from the expression \eqref {candidate of structure of N_0}.
On the other hand, we have
$
N_{rm-1}\subset
\ker\left(w|_{N_0}\right)
$
by Definition \ref {def: framing structure} (ii) and
$\length N_{rm-1}=r$ by Definition \ref {def: framing structure} (i).
Thus we have $N_{rm-1}=\ker\left(w|_{N_0}\right)$,
$\dim\ker\big(w|_{N_0}\big)=r$
and $s_r\geq 1$.

\noindent
\textbf{Claim A.} \, $s_{r-i+1}=i$ holds for $0<i< r-1$.

\noindent
{\it Proof of Claim A.}
It is sufficient to prove
$\#\{ k\,|\, s_k\leq i\}=i$ for $0<i<r-1$.
Note that $w^i\ker\left(w^{i+1}|_{N_0}\right)\subset \ker\left( w|_{N_0}\right)$
holds and \eqref {candidate of structure of N_0} implies
\[
\#\{ k\,|\, s_k\leq i\}=
\dim\ker\left( w|_{N_0}\right)-\dim \left( w^i\ker\big(w^{i+1}|_{N_0}\big)\right)
=r-\dim \left( w^i\ker\big(w^{i+1}|_{N_0}\big)\right).
\]
So we only have to prove
$\dim \left( w^i\ker\big(w^{i+1}|_{N_0}\big)\right)=r-i$
for $0<i<r-1$.
Take any element $h\in \ker\big(w^{i+1}|_{N_0}\big)$.
Then the condition~(iv') of Lemma \ref {Lem: structure of hom space} yields
\[
h\in\ker\big(w^{r-1}|_{N_0}\big)=N_{rm-r+1}\subset N_{rm-r}.
\]
So Definition \ref {def: framing structure} (ii) implies
$(w^ih)(1)=h(w^i)\in \overline{l}_i$  
and $(w^ih)(w^k)=(w^{i+1}h)(w^{k-1})=0$ hold
for $1\leq k\leq r-1$ because $w^{i+1}h=0$.
Then
$\dim \left( w^i\ker\big(w^{i+1}|_{N_0}\big)\right)
\leq\dim \overline{l}_i=r-i$
holds for $0<i<r-1$.
Since there are short exact sequences
\[
0 \longrightarrow \ker\left( w^k|_{N_0}\right)
\longrightarrow \ker\left( w^{k+1}|_{N_0}\right)
\longrightarrow w^k\ker\left(w^{k+1}|_{N_0}\right)
\longrightarrow 0
\]
for $1\leq k\leq i$, we have
\begin{align*}
\length \ker\left( w^{i+1}|_{N_0}\right)
&=
\sum_{k=0}^i\dim(w^k\ker(w^{k+1}|_{N_0}))
\\
&\leq 
\sum_{k=0}^i (r-k)
=\frac{r(r+1)}{2}-\frac{(r-i-1)(r-i)}{2}
=\length N_{rm-i-1}.
\end{align*}
Since
$N_{rm-i-1}\subset \ker\left( w^{i+1}|_{N_0}\right)$,
we have
$\dim \left( w^k\ker\big(w^{k+1}|_{N_0}\big)\right)=r-k$
for any $0<k\leq i$.
Thus Claim A is proved.
\hfill
$\square$

\noindent
\textbf{Claim B.} \, $s_2\leq r-1$.

\noindent
{\it Proof of Claim B.}
Suppose that $s_2\geq r$.
Then there are linearly independent elements
$h_1,h_2\in w^{r-1}N_0$ such that $wh_1=wh_2=0$.
Since $h_1(w^i)=(wh_1)(w^{i-1})=0$ for $1\leq i\leq r-1$,
$h_1$ is determined by $h_1(1)$ which belongs to
$z^{m-1}\overline{E}$ because of $zh_1=0$.
By Definition \ref{def: framing structure} (iii),
we have 
${\mathcal D}^{(r-1)}_{\overline{\nabla},\boldsymbol{\delta}}(h_1)=0$,
which leads to
\[
R\circ h_1=h_1\circ(a(w)-a_0(z))\mathrm{id}_{\mathcal A}
=(a_1(z)+a_2(z)w+\cdots+a_{r-1}(z)w^{r-2})wh_1
=0.
\]
So we have $h_1(1)\in\ker(R|_{z^{m-1}\overline{E}})$.
Similarly $h_2$ is determined by $h_2(1)$
which belongs to $\ker(R|_{z^{m-1}\overline{E}})$.
Since $\dim\ker(R|_{z^{m-1}\overline{E}})=1$
by Lemma \ref {Lem: alternative condition of FS},
$h_1(1),h_2(1)$ are linearly dependent.
This implies that $h_1,h_2$ are linearly dependent,
which contradicts the assumption.
Thus we have $s_2\leq r-1$.
\hfill $\square$

\noindent
Note that $s_1\leq rm$. 
By Claim A and Claim B, we have
\begin{align*}
\length N_0
=s_1+s_2+\cdots+s_r
&\leq
rm+r-1+\sum_{i=3}^{r} (r-i+1)
\\
&=rm+(r-1)+\frac{(r-1)(r-2)}{2}
=rm+\frac{r(r-1)}{2}.
\end{align*}
Since $\length N_0=rm+r(r-1)/2$,
the above inequality yields
$s_1=rm$ and $s_2=r-1$.
Thus the lemma is proved.
\end{proof}

\begin{Lem} \label {Lem: structure of Hom from (iv)}
Let $(\{N_i\},\{\overline{l}_i\})$ be a 
$(\nu(w),\boldsymbol{\delta})$-framing structure
on $(\overline{E},\overline{\nabla})$.
Then the condition \textup{(iv')} of Lemma \ref {Lem: structure of hom space} holds.
In particular, Lemma \ref{Lem: structure of hom space} yields
\[
N_0 \; \cong \;
\mathbb{C}[w]/(w^{rm})\oplus
\mathbb{C}[w]/(w^{r-1})\oplus\cdots\oplus\mathbb{C}[w]/(w).
\]
\end{Lem}

\begin{proof}
Since $w^{2r-2}N_0\neq 0$ automatically holds from
the condition (iv) of
Definition \ref {def: framing structure},
we only have to prove
$N_{rm-r+1}=\ker\big(w^{r-1}|_{N_0}\big)$.
As in the proof of Lemma \ref {Lem: structure of hom space},
we can write
\[
N_0\cong \mathbb{C}[w]/(w^{s_1})\oplus\cdots\oplus\mathbb{C}[w]/(w^{s_r})
\]
with 
$s_1\geq s_2\geq\cdots\geq s_r$.
Since $w^{rm-1}N_0\neq 0$ holds by Definition \ref{def: framing structure} (iv),
we have $s_1=rm$.
Then the equality
\[
s_2+\cdots+s_r=rm+\frac{1}{2}r(r-1)-s_1=\frac{1}{2}r(r-1)
\]
holds.
Since $N_{rm-r+1}\subset\ker\big(w^{r-1}|_{N_0}\big)$,
we have
\begin{align*}
\frac{1}{2}r(r+1)-1
&=
\length N_{rm-r+1}
\\
&\leq \length \ker\big(w^{r-1}|_{N_0}\big)
\\
&\leq
\length
\big( (w^{rm-r+1})/(w^{rm})\oplus\mathbb{C}[w]/(w^{s_2})\oplus
\cdots\oplus\mathbb{C}[w]/(w^{s_r})\big)
\\
&=
r-1+s_2+\cdots+s_r=r-1+\frac{1}{2}r(r-1)=\frac{1}{2}r(r+1)-1.
\end{align*}
Thus we have $N_{rm-r+1}=\ker\big(w^{r-1}|_{N_0}\big)$.
\end{proof}

\begin{Lem} \label {Lem: existence of isomorphism}
Let $(\{N_i\},\{\overline{l}_i\})$ be a 
$(\nu(w),\boldsymbol{\delta})$-framing structure
on $(\overline{E},\overline{\nabla})$.
Then there exists an isomorphism
$\sigma\colon {\mathcal A}\xrightarrow{\sim} \overline{E}$
which belongs to $N_0$.
\end{Lem}

\begin{proof}
By Definition \ref {def: framing structure} (iv),
there exists $\sigma\in N_0$ such that
$w^{mr-1}\sigma\neq 0$.
Recall the proof of Lemma \ref {Lem: alternative condition of FS}.
Setting $h':=z^{m-2}w\sigma$, we have $w^{2r-2}h'=w^{rm-1}\sigma\neq 0$.
From the proof of Lemma \ref {Lem: alternative condition of FS},
we have $R^{r-1}((z^{m-1}\sigma)(1))\neq 0$
and $R^r((z^{m-1}\sigma)(1))=0$.
This implies that
$(z^{m-1}\sigma)(1),R((z^{m-1}\sigma)(1)),\ldots,R^{r-1}((z^{m-1}\sigma)(1))$
are linearly independent.
Since
$
 R^k((z^{m-1}\sigma(1)))=z^{m-1}\sigma((a(w)-a_0(z))^k)
$
holds,
$z^{m-1}\sigma\colon {\mathcal A}\rightarrow z^{m-1}\overline{E}$
is surjective. 
So the induced linear map
${\mathcal A}/z{\mathcal A}\ni\overline{a}\mapsto z^{m-1}\sigma(a)
\in z^{m-1}\overline{E}$
is an isomorphism, because
$\dim_{\mathbb{C}}{\mathcal A}/z{\mathcal A}
=\dim_{\mathbb{C}} z^{m-1}\overline{E}=r$.
This is equivalent to the bijectivity of
$\sigma\otimes\mathbb{C}[z]/(z)\colon
{\mathcal A}/z{\mathcal A}\rightarrow \overline{E}|_x$.
Thus $\sigma\colon {\mathcal A}\rightarrow \overline{E}$
is an isomorphism.
\end{proof}

\begin{Lem} \label {Lem: recover of framing structure}
Let $(\{N_i\},\{\overline{l}_i\})$ be a 
$(\nu(w),\boldsymbol{\delta})$-framing structure
on $(\overline{E},\overline{\nabla})$.
Take an isomorphism 
$\sigma\colon {\mathcal A}\xrightarrow{\sim}\overline{E}$
which belongs to $N_0$.
Then we have
\begin{equation}
\label {recover of N_{rm-r+j}}
N_{rm-r+j}
=
\left\{
h\in \Hom_{{\mathcal O}_{mx}}({\mathcal A},\overline{E})
\,\middle|\,
\text{\rm
$h(w^k)\in \overline{l}_k$ for $0\leq k\leq r-1$ and $w^{r-j}h=0$}
\right\}
\end{equation}
for $0\leq j\leq r-1$
and
\begin{align}
\label {recover of N_i}
\begin{split}
N_i
&=
{\mathcal A} w^i\sigma+N_{rm-r}
=
\left\{
h\in\Hom_{{\mathcal O}_{mx}}({\mathcal A},\overline{E})
\,\middle|
\begin{array}{l}
\exists \alpha(w)\in w^i{\mathcal A}\;\text{\rm such that}
\\
(h-\sigma\circ\alpha(w)\mathrm{id}_{\mathcal A})
(w^k{\mathcal A})\subset \overline{l}_k
\\
\text{\rm for $0\leq k\leq r-1$}
\end{array}
\right\}
\end{split}
\end{align}
for $0\leq i\leq rm-r$.
\end{Lem}

\begin{proof}
By Definition \ref {def: framing structure} (ii),
$N_{rm-r+j}$ is contained in the right-hand side
of \eqref {recover of N_{rm-r+j}}.
The length of the right-hand side of \eqref{recover of N_{rm-r+j}} is
\[
\sum_{k=0}^{r-j-1} \dim \overline{l}_k
=
\sum_{k=0}^{r-1}(r-k)-\sum_{k=r-j}^{r-1}(r-k)
=\frac{1}{2}r(r+1)-\frac{1}{2}j(j+1),
\]
which is equal to $\length N_{rm-r+j}$
by Definition \ref {def: framing structure} (i).
Thus the equality \eqref{recover of N_{rm-r+j}} holds.

We have ${\mathcal A} w^i \sigma +N_{rm-r}\subset N_i$,
for $0\leq i\leq rm-r$.
Since $N_{rm-r}\subset \ker(w^r|_{N_0})$, we can see
${\mathcal A}w^i\sigma\cap N_{rm-r} \subset  {\mathcal A} z^{m-1}\sigma$.
By using $\length ( {\mathcal A} z^{m-1}\sigma)=r$,
we have
\begin{align*}
\length ({\mathcal A}w^i\sigma+N_{rm-r})
&=
\length ({\mathcal A} w^i\sigma)
+
\length(N_{rm-r})
-
\length({\mathcal A}w^i\sigma\cap N_{rm-r})
\\
&\geq
rm-i+\frac{1}{2}r(r+1)-r
=
\length N_i.
\end{align*}
Thus we have $N_i={\mathcal A} w^i\sigma+N_{rm-r}$.
The second equality of \eqref{recover of N_i}
follows from \eqref{recover of N_{rm-r+j}}.
\end{proof}

\begin{Prop} \label {natural representation matrix}
Let $(\{N_i\},\{\overline{l}_i\})$ be a 
$(\nu(w),\boldsymbol{\delta})$-framing structure on 
$(\overline{E},\overline{\nabla})$.
Take an isomorphism
$\sigma\colon{\mathcal A}\xrightarrow{\sim}\overline{E}$ belonging to $N_0$,
whose existence is due to Lemma \ref {Lem: existence of isomorphism}.
Then the representation matrix of $\overline{\nabla}$
with respect to the basis
$\sigma(1),\ldots,\sigma(w^{r-1})$ is of the form
\[
\begin{pmatrix}
\nu_0(z)+\delta_0\mathrm{d}z/z & z\nu_{r-1}(z) & \cdots & z\nu_1(z) \\
\nu_1(z)+b_{10}\mathrm{d}z/z & \nu_0(z)+\delta_1\mathrm{d}z/z & \cdots & z\nu_2(z) \\
\vdots & & \ddots & \vdots \\
\nu_{r-1}(z)+b_{r-1,0}\mathrm{d}z/z & \nu_{r-2}(z)+b_{r-1,1}\mathrm{d}z/z 
& \cdots & \nu_0(z)+\delta_{r-1}\mathrm{d}z/z
\end{pmatrix},
\]
where $b_{ij}\in\mathbb{C}$ for $i> j$.
\end{Prop}

\begin{proof}
By Definition \ref {def: framing structure} (ii),
$z^{m-1}\sigma\in N_{rm-r}$
and
$z^{m-1}\sigma(w^i)\in \overline{l}_i$ hold for $0\leq i\leq r-1$.
So $\overline{l}_i$ is generated by 
$z^{m-1}\sigma(w^i),\ldots,z^{m-1}\sigma(w^{r-1})$,
because these are linearly independent 
from the proof of Lemma \ref{Lem: existence of isomorphism}.
Since ${\mathcal D}^{(i)}_{\overline{\nabla},\boldsymbol{\delta}}(w^i\sigma)(1)
\in \overline{l}_{i+1}\otimes\overline{\Omega}(mx)$
for $0\leq i\leq r-1$
by Definition \ref {def: framing structure} (iii),
$\overline{\nabla}(\sigma(w^i))
\equiv \sigma\big(w^i\nu(w)^{(i)}_{\boldsymbol{\delta}}\big)
\pmod {\overline{l}_{i+1}\otimes\overline{\Omega}(mx)}$
holds.
So there exist $b_{i+1,i},\ldots,b_{r-1,i}\in\mathbb{C}$ such that
\begin{align}
\label {image of w^i}
\begin{split}
\overline{\nabla}(\sigma(w^i))
&=
\sigma\Big(w^i\Big(\nu_0(z)+\nu_1(z)w+\cdots
+\nu_{r-1}(z)w^{r-1}+\delta_i\frac{\mathrm{d}z}{z}
\Big)\Big)
\\
&\qquad
+b_{i+1,i}\frac{\mathrm{d}z}{z}\sigma(w^{i+1})+\cdots+b_{r-1,i}\frac{\mathrm{d}z}{z}\sigma(w^{r-1})
\\
&=
\Big(\nu_0(z)+\delta_i\frac{\mathrm{d}z}{z}\Big)\sigma(w^i)+\cdots
+\Big(\nu_{r-i-1}(z)+b_{r-1,i}\frac{\mathrm{d}z}{z}\Big)\sigma(w^{r-1})
\\
&\qquad
+\nu_{r-i}(z)z\sigma(1)+\cdots+\nu_{r-1}(z)z\sigma(w^{i-1})
\end{split}
\end{align}
holds.
The equality  \eqref{image of w^i}
gives the desired representation matrix of $\overline{\nabla}$.
\end{proof}

\begin{Prop} \label {Prop: replacement of diagonal part}
Let $(\{N_i\},\{\overline{l}_i\})$ be a 
$(\nu(w),\boldsymbol{\delta})$-framing structure
on $(\overline{E},\overline{\nabla})$.
Take a tuple of complex numbers
$\boldsymbol{\delta'}=(\delta'_0,\ldots,\delta'_{r-1})$
satisfying $\delta'_0+\cdots+\delta'_{r-1}=0$.
Then there exists an isomorphism
$g\colon \overline{E}\xrightarrow{\sim}\overline{E}$
such that 
$g-\mathrm{id}_{\overline{E}}$ maps
$\overline{E}$ to $z^{m-1}\overline{E}$
and 
$(\{ g(N_i) \}, \{\overline{l}_i\})$
becomes a $(\nu(w),\boldsymbol{\delta'})$-framing structure
on $(\overline{E},\overline{\nabla})$.
\end{Prop}

\begin{proof}
By Lemma \ref{Lem: existence of isomorphism},
there exists an isomorphism 
$\sigma\colon {\mathcal A}\xrightarrow{\sim}\overline{E}$
which belongs to $N_0$.
Then 
$\sigma(1),\sigma(w),\ldots,\sigma(w^{r-1})$
becomes a basis of $\overline{E}$
over ${\mathcal O}_{mx}$.
Fix an integer $i$ with $0\leq i\leq r-1$.
Take $b_i\in\mathbb{C}$.
Define a homomorphism
$
\gamma_{b_i}\colon \overline{E}\longrightarrow z^{m-1}\overline{E}
$
by
\[
 \gamma_{b_i}\left(\sigma(w^k)\right)=
\begin{cases}
0 & (\text{if $k\neq i+1$}) \\
b_i z^{m-1}\sigma(w^{i}) & (\text{if $k=i+1$}).
\end{cases}
\]
Put $\alpha_i:=b_ia_1(0)$. 
Then we have
\begin{align}
\label {D^i for w^i}
\begin{split}
{\mathcal D}_{\overline{\nabla},\boldsymbol{\delta}}^{(i)}
(\gamma_{b_i}\circ \sigma) (w^i)
&=
\overline{\nabla}(\gamma_{b_i}(\sigma(w^{i})))
-\gamma_{b_i}(\sigma(\nu(w)w^i))
\\
&=
-\gamma_{b_i}\big(
\sigma\big(\nu_0(z)w^i+\cdots+\nu_{r-1}(z)w^{r+i-1}\big)\big)
\\
&=
-b_iz^{m-1}\nu_1(z)\sigma(w^i)
=
-\alpha_i\frac{\mathrm{d}z}{z}\sigma(w^i)
\end{split}
\end{align}
and
\begin{align}
\label {D^{i+1} for w^{i+1}}
\begin{split}
{\mathcal D}_{\overline{\nabla},\boldsymbol{\delta}}^{(i+1)}
(\gamma_{b_i}\circ \sigma)(w^{i+1})
&=
\overline{\nabla}(\gamma_{b_i}(\sigma(w^{i+1})))
-\gamma_{b_i}(\sigma(\nu(w)w^{i+1}))
\\
&=
\overline{\nabla}(b_iz^{m-1}\sigma(w^i))-b_iz^{m-1}\nu_0(z)\sigma(w^i)
\\
&=
b_iz^{m-1}\sigma(\nu(w)w^i)-b_iz^{m-1}\sigma(\nu_0(z)w^i)
\\
&=
b_iz^{m-1}\sigma\big((\nu_1(z)w+\cdots+\nu_{r-1}(z)w^{r-1})w^i\big)
\\
&\equiv
\alpha_i\frac{\mathrm{d}z}{z}\sigma(w^{i+1})
\pmod{\overline{l}_{i+2}\otimes\overline{\Omega}(mx)}.
\end{split}
\end{align}
For $0\leq k\leq i-1$, 
\begin{align}
\label {D^k for w^k with k<i}
\begin{split}
{\mathcal D}_{\overline{\nabla},\boldsymbol{\delta}}^{(k)}
(\gamma_{b_i}\circ \sigma)(w^k)
&=
\overline{\nabla}(\gamma_{b_i}(\sigma(w^k)))
-\gamma_{b_i}(\sigma(\nu(w)w^k))
\\
&=
-\gamma_{b_i}\left(\sigma\left(\nu_0(z)w^k+\cdots
+\nu_{r-1}(z)w^{r+k-1}\right)\right)
\\
&=
-b_iz^{m-1}\nu_{i-k+1}(z)\sigma(w^i)
\in \overline{l}_i\otimes\overline{\Omega}(mx)
\end{split}
\end{align}
holds and for $i+2\leq k\leq r-1$,
\begin{align}
\label {D^k for w^k with k>i+1}
\begin{split}
{\mathcal D}^{(k)}_{\overline{\nabla},\boldsymbol{\delta}}
(\gamma_{b_i}\circ \sigma)(w^k)
&=
\overline{\nabla}(\gamma_{b_i}(\sigma(w^k)))
-\gamma_{b_i}(\sigma(\nu(w)w^k))
\\
&=
-\gamma_{b_i}\left(\sigma\left(\nu_0(z)w^k+\cdots
+\nu_{r-1}(z)w^{r+k-1}\right)\right)
=0
\end{split}
\end{align}
holds.
Set 
\[
\boldsymbol{\tilde{\delta}}=
(\tilde{\delta}_0,\ldots,\tilde{\delta}_{r-1})
:=
(\delta_0,\ldots,\delta_i-\alpha_i,\delta_{i+1}+\alpha_i,\ldots,\delta_{r-1}).
\]
Then we have
\begin{align*}
{\mathcal D}^{(k)}_{\overline{\nabla},\boldsymbol{\tilde{\delta}}}
\left((\mathrm{id}+\gamma_{b_i})\circ \sigma \right)(w^k)
&=
{\mathcal D}^{(k)}_{\overline{\nabla},\boldsymbol{\tilde{\delta}}}
(\sigma)(w^k)
+
{\mathcal D}^{(k)}_{\overline{\nabla},\boldsymbol{\tilde{\delta}}}
(\gamma_{b_i}(\sigma))(w^k)
\\
&=
{\mathcal D}^{(k)}_{\overline{\nabla},\boldsymbol{\delta}}
(\sigma)(w^k)
-(\tilde{\delta}_k-\delta_k)\sigma(w^k)\frac{\mathrm{d}z}{z}
+
{\mathcal D}^{(k)}_{\overline{\nabla},\boldsymbol{\delta}}
(\gamma_{b_i}(\sigma))(w^k)
\end{align*}
for $0\leq k\leq r-1$.
By using \eqref{D^i for w^i}, \eqref{D^{i+1} for w^{i+1}},
\eqref{D^k for w^k with k<i} and \eqref {D^k for w^k with k>i+1},
we have
\[
{\mathcal D}^{(k)}_{\overline{\nabla},\boldsymbol{\tilde{\delta}}}
\left( (\mathrm{id}+\gamma_{b_i})\circ \sigma\right)(w^k)
\in
\overline{l}_{k+1}\otimes\overline{\Omega}(mx)
\]
for $0 \leq k\leq r-1$.
Since $N_0$ is expressed by
Lemma \ref{Lem: recover of framing structure},
we can see
\[
{\mathcal D}^{(k)}_{\overline{\nabla},\boldsymbol{\tilde{\delta}}}
\left( (\mathrm{id}+\gamma_{b_i})w^kN_0\right)
\subset
\Hom_{{\mathcal O}_{mx}}({\mathcal A},\overline{l}_{k+1})
\otimes\overline{\Omega}(mx)
\]
for $0\leq k\leq r-1$,
which means that
$(\{(\mathrm{id}+\gamma_{b_i})(N_k)\},\{\overline{l}_k\})$
is a $(\nu(w),\boldsymbol{\tilde{\delta}})$-framing structure.

Composing 
$\mathrm{id}+\gamma_{b_0},\ldots,\mathrm{id}+\gamma_{b_{r-1}}$
for a suitable choice of $b_0,\ldots,b_{r-1}$,
we get the desired isomorphism.
\end{proof}

The following corollary follows from
Proposition \ref{Prop: replacement of diagonal part}
and
Proposition \ref{natural representation matrix}.

\begin{Cor} \label {cor transformed representation matrix}
Under the same assumption as Proposition \ref {Prop: replacement of diagonal part},
there are an isomorphism 
$g\colon \overline{E}\xrightarrow{\sim}\overline{E}$
of the form $g=\mathrm{id}+z^{m-1}\gamma$
and an isomorphism 
$\sigma\colon{\mathcal A}\xrightarrow{\sim}\overline{E}$
belonging to $N_0$ 
such that
$\overline{\nabla}$ is represented by a matrix of the form
\[
\begin{pmatrix}
\nu_0(z)+\delta'_0\mathrm{d}z/z & z\nu_{r-1}(z) & \cdots & z\nu_1(z) \\
\nu_1(z)+b'_{10}\mathrm{d}z/z & \nu_0(z)+\delta'_1\mathrm{d}z/z & \cdots & z\nu_2(z) \\
\vdots & & \ddots & \vdots \\
\nu_{r-1}(z)+b'_{r-1,0}\mathrm{d}z/z & \nu_{r-2}(z)+b'_{r-1,1}\mathrm{d}z/z 
& \cdots & \nu_0(z)+\delta'_{r-1}\mathrm{d}z/z
\end{pmatrix},
\]
with respect to the basis
$(g\circ \sigma)(1),(g\circ \sigma)(w),\ldots,(g\circ \sigma)(w^{r-1})$
of $\overline{E}$.
\end{Cor}

\begin{Prop} \label {Prop: (nu,c)-framing structure and (nu,n)-framing structure}
Let $\boldsymbol{\delta}=(\delta_0,\ldots,\delta_{r-1})$ and
$\boldsymbol{\delta'}=(\delta'_0,\ldots,\delta'_{r-1})$
be tuples of complex numbers satisfying
$
\delta_0+\cdots+\delta_{r-1}=\delta'_0+\cdots+\delta'_{r-1}=0.
$
Then there is a bijection between the set of
$(\nu(w),\boldsymbol{\delta})$-framing structures on 
$(\overline{E},\overline{\nabla})$
and the set of $(\nu(w),\boldsymbol{\delta'})$-framing structures on
$(\overline{E},\overline{\nabla})$
through the correspondence
\[
 (\{N_i\},\{\overline{l}_i\}) \;\longmapsto \; (\{g(N_i)\},\{\overline{l}_i\})
\]
for an isomorphism 
$g\colon \overline{E}\xrightarrow{\sim}\overline{E}$
satisfying $g(\overline{l}_i)=\overline{l}_i$ for any $i$.
Moreover, $(\{g(N_i)\},\{\overline{l}_i\})$
is independent of the choice of the isomorphism $g$.
\end{Prop}

\begin{proof}
Since we know Proposition \ref {Prop: replacement of diagonal part},
we only have to prove that
$g(N_i)$ is independent of $g$.
So let $g_1,g_2\colon \overline{E}\xrightarrow{\sim}\overline{E}$
be two isomorphisms such that
$g_1(\overline{l}_i)=g_2(\overline{l}_i)=\overline{l}_i$ and both
$(\{g_1(N_i)\},\{\overline{l}_i\})$ and $(\{g_2(N_i)\},\{\overline{l}_i\})$
are $(\nu(w),\boldsymbol{\delta'})$-framing structures.
We want to prove $g_1(N_i)=g_2(N_i)$.
Set $N'_i:=g_1(N_i)$ and $\tau:=g_2\circ g_1^{-1}$.
Then it is sufficient to prove
$\tau(N'_i)=N'_i$ for any $i$.
Take an isomorphism 
$\sigma\colon{\mathcal A}\xrightarrow{\sim}\overline{E}$
which belongs to $N'_0$.
Since $(\{\tau(N'_i)\},\{\overline{l}_i\})$ is a 
$(\nu(w),\boldsymbol{\delta'})$-framing structure 
and $\tau\circ \sigma\in\tau(N'_0)$, we have
\begin{align}
\label {condition of framing structure for sigma}
\begin{split}
 \big(\overline{\nabla}\circ \tau\circ \sigma
-\tau\circ \sigma\circ\nu(w)^{(k)}_{\boldsymbol{\delta'}}
\mathrm{id}_{\mathcal A}\big)
(w^k{\mathcal A})
&=
{\mathcal D}^{(k)}_{\overline{\nabla},\boldsymbol{\delta'}}
(\tau\circ \sigma)(w^k{\mathcal A})
\\
&\subset \overline{l}_{k+1}\otimes\overline{\Omega}(mx)
\end{split}
\end{align}
for $0\leq k\leq r-1$.
Then we have
$\overline{\nabla}\circ \tau \circ \sigma|_{z{\mathcal A}}
=\tau\circ \sigma\circ \nu(w)\mathrm{id}_{\mathcal A}|_{z{\mathcal A}}$
by restricting \eqref {condition of framing structure for sigma}
to $z{\mathcal A}$.
Composing with $\sigma^{-1}$, we have
\begin{equation}
\label {equation from compatibility for sigma}
\sigma^{-1}\circ\overline{\nabla}\circ \tau \circ \sigma|_{z{\mathcal A}}
=
\sigma^{-1}\circ\tau\circ \sigma\circ
\nu(w)\mathrm{id}_{\mathcal A}|_{z{\mathcal A}}.
\end{equation}
Since $\sigma\in N'_0$, we have
\begin{equation}
\label {substitute sigma to compatibility equation}
(\overline{\nabla}\circ \sigma
-\sigma\circ\nu(w)^{(k)}_{\boldsymbol{\delta'}}\mathrm{id}_{\mathcal A})
(w^k{\mathcal A})
=
{\mathcal D}^{(k)}_{\overline{\nabla},\boldsymbol{\delta'}}
(\sigma)(w^k{\mathcal A})
\subset \overline{l}_{k+1}\otimes\overline{\Omega}(mx)
\end{equation}
for $0\leq k\leq r-1$.
Restricting to $z{\mathcal A}$ implies
$\overline{\nabla}\circ \sigma|_{z{\mathcal A}}
=\sigma\circ \nu(w)\mathrm{id}_{\mathcal A}|_{z{\mathcal A}}$.
Applying $\sigma^{-1}$ from the left and 
$\sigma^{-1}\circ\tau\circ \sigma$
from the right, we have
\begin{equation}
\label {compose sigma to compatibility equation}
\sigma^{-1}\circ\overline{\nabla}\circ \tau\circ \sigma|_{z{\mathcal A}}
=
\nu(w)\mathrm{id}_{\mathcal A}\circ \sigma^{-1}
\circ\tau\circ \sigma|_{z{\mathcal A}}.
\end{equation}
Combining \eqref{equation from compatibility for sigma} and 
\eqref{compose sigma to compatibility equation}
yields the equality
\[
\sigma^{-1}\circ\tau\circ \sigma\circ
\nu(w)\mathrm{id}_{\mathcal A}|_{z{\mathcal A}}
=
\nu(w)\mathrm{id}_{\mathcal A}\circ \sigma^{-1}
\circ\tau\circ \sigma|_{z{\mathcal A}},
\]
which is equivalent to
\begin{equation}
\label {commutativity with a(w)-a_0(z)}
\sigma^{-1}\circ\tau\circ \sigma\circ
(a(w)-a_0(z))\mathrm{id}_{\mathcal A}|_{z{\mathcal A}}
=
(a(w)-a_0(z))\mathrm{id}_{\mathcal A}\circ \sigma^{-1}
\circ\tau\circ \sigma|_{z{\mathcal A}}.
\end{equation}
Since $a(w)-a_0(z)=a_1(z)w+\cdots+a_{r-1}(z)w^{r-1}$
with $a_1(z)\in{\mathcal O}_{mx}^{\times}$,
we can express $w$ as a polynomial in $a(w)-a_0(z)$.
So \eqref {commutativity with a(w)-a_0(z)}
implies that 
$
\sigma^{-1}\circ\tau\circ \sigma|_{z{\mathcal A}}:
z{\mathcal A}
\longrightarrow
z{\mathcal A}
$
is a homomorphism of ${\mathcal A}$-modules and we can write
$
\sigma^{-1}\circ\tau\circ \sigma|_{z{\mathcal A}}
=
\alpha(w)\mathrm{id}_{z{\mathcal A}}
$
for some $\alpha(w)\in{\mathcal A}$.
Set
\[
\rho:=\tau-\sigma\circ\alpha(w)\mathrm{id}_{\mathcal A}\circ \sigma^{-1}.
\]
Then the image of $\rho$ is contained in $z^{m-1}\overline{E}$.
In order to prove $\tau(N'_i)=N'_i$,
it is sufficient to prove $\rho(N'_i)\subset N'_i$,
because
$(\sigma\circ\alpha(w)\mathrm{id}_{\mathcal A}\circ \sigma^{-1})(N'_i)=N'_i$
by Lemma \ref {Lem: recover of framing structure}.
Since $\sigma\circ\alpha(w)\mathrm{id}_{\mathcal A}\in N'_0$,
we have
\begin{align*}
(\overline{\nabla}\circ \sigma\circ\alpha(w)\mathrm{id}_{\mathcal A}
-\sigma\circ\alpha(w)\nu(w)^{(k)}_{\boldsymbol{\delta'}}
\mathrm{id}_{\mathcal A})
(w^k{\mathcal A})
&=
{\mathcal D}^{(k)}_{\overline{\nabla},\boldsymbol{\delta'}}
(\sigma\circ\alpha(w)\mathrm{id}_{\mathcal A})(w^k{\mathcal A})
\\
&\subset \overline{l}_{k+1}\otimes\overline{\Omega}(mx).
\end{align*}
Combining with
\eqref{condition of framing structure for sigma},
we have
\begin{equation}
\label {substitute rho to compatibility equation}
(\overline{\nabla}\circ \rho\circ \sigma
-\rho\circ \sigma\circ 
\nu(w)^{(k)}_{\boldsymbol{\delta'}}\mathrm{id}_{\mathcal A})
(w^k{\mathcal A})
\subset
\overline{l}_{k+1}\otimes\overline{\Omega}(mx)
\end{equation}
for $0\leq k\leq r-1$.
Composing $\rho$ to \eqref{substitute sigma to compatibility equation},
we have
\[
(\rho\circ\overline{\nabla}\circ  \sigma
-\rho\circ \sigma\circ\nu(w)^{(k)}_{\boldsymbol{\delta'}}
\mathrm{id}_{\mathcal A})
(w^k{\mathcal A})
\subset
\overline{l}_{k+1}\otimes\overline{\Omega}(mx)
\]
for $0\leq k\leq r-1$.
Combining this with \eqref{substitute rho to compatibility equation},
we have
\[
(R\circ\rho-\rho\circ R)(\sigma(w^k))\in \overline{l}_{k+1}
\]
for $0\leq k\leq r-1$,
where $R$ is given by \eqref{restriction of traceless connection}.
By Proposition \ref{natural representation matrix},
the representation matrix of $R|_x$ 
with respect to the basis
$\sigma(1),\ldots,\sigma(w^{r-1})$ is 
\[
\begin{pmatrix}
0 & 0 & \cdots & 0 \\
a_1(0) & \ddots &  & \vdots \\
\vdots & \ddots &  &  \\
a_{r-1}(0) & \cdots
& a_1(0) & 0
\end{pmatrix}.
\]
Since $a_1(0)\neq 0$, we can see that
$\rho$ is represented by a lower triangular matrix
with respect to the same basis.
This means $\rho(\sigma(w^k))\in \overline{l}_k$ for $0\leq k\leq r-1$
and we have $\rho(N'_0)\subset N'_{rm-r}$
by Lemma \ref {Lem: recover of framing structure}.
Thus we have $\rho(N'_i)\subset N'_i$ for any $i$.
\end{proof}

\subsection{Relationship with generic ramified singularities}
\label {Subsect relation with ram conn}

The aim of this subsection is to prove the equivalence 
between framing structures of Definition~\ref{def: framing structure} 
and ramified structures introduced in \cite{Inaba3}.

The notations $X$, $x\in X$,
$z\in\widehat{\mathcal O}_{X,x}=\mathbb{C}[[z]]$,
$\widehat{\Omega}(mx)$, $\overline{\Omega}(mx)$,
${\mathcal A}$,
$a_i(T)$, $\nu_k(z)$ and $\nu(w)$
are the same as in Section \ref{subsect framing str}.
Set 
\[
 \nu'(w):=\nu(w)-\frac{(r-1)\mathrm{d}z}{2rz}.
\]

Let $\widehat{E}$ be a free $\widehat{\mathcal O}_{X,x}$-module
of rank $r$ and let
$\widehat{\nabla}\colon\widehat{E}
\rightarrow\widehat{E}\otimes\widehat{\Omega}(mx)$
be a formal connection, that is, a $\mathbb{C}$-linear map satisfying
$\widehat{\nabla}(fa)=a\otimes \mathrm{d}f+f\widehat{\nabla}(a)$
for $f\in\widehat{\mathcal O}_{X,x}$ and $a\in\widehat{E}$.
As a connection with ramified singularities,
we consider the following.

\begin{Def} \label {Def: generic ramified connection}
We call $(\widehat{E},\widehat{\nabla})$ a \emph{generic 
$\nu'(w)$-ramified connection}
if there is a formal isomorphism
$
(\widehat{E},\widehat{\nabla})
\cong
(\mathbb{C}[[w]],\mathrm{d}+\nu'(w)),
$
where $\mathbb{C}[[w]]$ is regarded as a free $\mathbb{C}[[z]]$-module
of rank $r$.
\end{Def}

The following notion of a ramified structure characterizes
the formal connection $(\widehat{E},\widehat{\nabla})$ satisfying
Definition \ref{Def: generic ramified connection}.

\begin{Def} 
[{\cite[Definitions 1.2 and 2.1]{Inaba3}} and {\cite[Definition 2.4]{Inaba4}}]
\label {Def: ramified structure}
Let $(\overline{E},\overline{\nabla})$ be a pair consisting of
a free ${\mathcal O}_{mx}$-module $\overline{E}$ of rank $r$
and a homomorphism
$\overline{\nabla}\colon\overline{E}\rightarrow\overline{E}\otimes\overline{\Omega}(mx)$
of $\mathcal{O}_{mx}$-modules.
Set $\bar{\mathcal A}:={\mathcal A}/z^{m-1}w{\mathcal A}$.
A \emph{$\nu'(w)$-ramified structure on $(\overline{E},\overline{\nabla})$}
consists of a tuple
$((V_k,L_k,\pi_k)_{0\leq k\leq r-1},(\phi_k)_{0\leq k\leq r})$
which satisfies the following:
\begin{itemize}
\item[\textup{(i)}]
$\overline{E}=V_0\supset V_1\supset\cdots\supset V_r=zV_0$
is a filtration by ${\mathcal O}_{mx}$-submodules
which satisfies $\length(V_k/V_{k+1})=1$
and $\overline{\nabla}(V_k)\subset V_k\otimes\overline{\Omega}(mx)$
for $0\leq k\leq r-1$,
\item[\textup{(ii)}]
$\pi_k\colon V_k\otimes\bar{\mathcal A}
\rightarrow L_k$ is a quotient free $\bar{\mathcal A}$-module
of rank one, such that the restriction
$\pi_k|_{V_k}\colon V_k\rightarrow L_k$ is surjective for $0\leq k\leq r-1$
and the diagram
\[
\begin{CD}
V_k\otimes\bar{\mathcal A}
@>\pi_k>> L_k
\\
@V \overline{\nabla} VV  @VV \nu'(w)+\frac{k \mathrm{d} z}{rz} V
\\
V_k\otimes\overline{\Omega}(mx)\otimes
\bar{\mathcal A}
@>\pi_k>> L_k\otimes\overline{\Omega}(mx)
\end{CD}
\]
is commutative for $0\leq k\leq r-1$,
\item[\textup{(iii)}]
for $0\leq k\leq r-1$,
$\phi_k\colon L_k\rightarrow L_{k-1}$
and
$\phi_r\colon (z)/(z^{m+1})\otimes L_0\rightarrow L_{r-1}$
are $\bar{\mathcal A}$-homomorphisms  such that 
$\mathrm{Im}(\phi_k)=wL_{k-1}$, $\mathrm{Im}(\phi_{r-1})=wL_{r-1}$,
the diagrams
\[
\begin{CD}
V_k\otimes\bar{\mathcal A} 
@> \pi_k >> L_k
\\
@VVV  @VV \phi_k V
\\
V_{k-1}\otimes\bar{\mathcal A} 
@>\pi_{k-1}>> L_{k-1}
\end{CD}
\]
are commutative for $1\leq k\leq r-1$
and the diagram
\[
\begin{CD}
(z)/(z^{m+1})\otimes V_0\otimes\bar{\mathcal A} 
@> 1\otimes \pi_0 >> (z)/(z^{m+1})\otimes L_0
\\
@VVV  @VV \phi_r V
\\
V_{r-1}\otimes\bar{\mathcal A} 
@>\pi_{r-1}>> L_{r-1}
\end{CD}
\]
is commutative,
\item[\textup{(iv)}]
there is an isomorphism
$\psi_k\colon L_k\xrightarrow{\sim} 
\big(w{\mathcal A}/w^{mr-r+2}{\mathcal A}\big)\otimes L_{k-1}$
of $\bar{\mathcal A}$-modules for $1\leq k\leq r-1$
such that the composition
\[
L_k\xrightarrow[\sim]{\psi_k} 
\big(w{\mathcal A}/w^{mr-r+2}{\mathcal A}\big)\otimes L_{k-1}
\rightarrow wL_{k-1}
\]
coincides with $\phi_k$ and the composition
\begin{align*}
& \big(z{\mathcal A}/z^mw{\mathcal A}\big)\otimes L_0
\xrightarrow{\phi_r} L_{r-1}
\xrightarrow[\sim]{\psi_{r-1}} 
\big(w{\mathcal A}/w^{mr-r+2}{\mathcal A}\big)\otimes L_{r-2}
\\
&\xrightarrow[\sim]{\psi_{r-2}}\cdots
\xrightarrow[\sim]{\psi_1}
\big(w{\mathcal A}/w^{mr-r+2}{\mathcal A}\big)^{\otimes r-1}
\otimes L_0
\xrightarrow{\sim}
\big(w^{r-1}{\mathcal A}/w^{mr}{\mathcal A}\big)\otimes L_0
\end{align*}
coincides with the $\bar{\mathcal A}$-homomorphism
obtained by tensoring $L_0$ to the canonical map
$z{\mathcal A}/w^{mr+1}{\mathcal A}\rightarrow 
w^{r-1}{\mathcal A}/w^{mr}{\mathcal A}$
\end{itemize}
\end{Def}

\begin{Rem} 
\leavevmode
\begin{itemize}
\item[\textup{(1)}]
Let $(\widehat{E},\widehat{\nabla})$ be a generic 
$\nu'(w)$-ramified connection.
Then the restriction $(\widehat{E}|_{mx},\widehat{\nabla}|_{mx})$
has a $\nu'(w)$-ramified structure
via the formal isomorphism
$\theta\colon(\widehat{E},\widehat{\nabla})
\xrightarrow{\sim}(\mathbb{C}[[w]],\mathrm{d}+\nu'(w))$
in Definition \ref{Def: generic ramified connection}.
Indeed, the induced $\nu'(w)$-ramified structure 
is formed by the modules
\[
V_k:=\theta^{-1}(w^k\mathbb{C}[[w]])/z^m\widehat{E},
\quad
L_k:=w^k\mathbb{C}[[w]]/w^{k+1}z^{m-1}\mathbb{C}[[w]]
\quad (0\leq k\leq r-1)
\]
together with the canonically induced homomorphisms $\pi_k$ and $\phi_k$ on them.

\item[\textup{(2)}]
Conversely, if we are given a formal connection
$(\widehat{E},\widehat{\nabla})$ whose restriction
$(\widehat{E}|_{mx},\widehat{\nabla}|_{mx})$ 
has a $\nu'(w)$-ramified structure,
then \cite[Proposition 1.3]{Inaba3} ensures that
$(\widehat{E},\widehat{\nabla})$
is a generic $\nu'(w)$-ramified connection.
Note that the condition~\textup{(iv)} of Definition \ref {Def: ramified structure}
is indispensable in this implication; otherwise, 
the correct formal type cannot be recovered
(see \cite[Appendix]{Inaba3}).
\end{itemize}
\end{Rem}

Let $(\overline{E},\overline{\nabla})$ be a pair consisting of
a free ${\mathcal O}_{mx}$-module $\overline{E}$ of rank $r$
and a homomorphism
$\overline{\nabla}\colon\overline{E}\rightarrow\overline{E}\otimes\overline{\Omega}(mx)$
of $\mathcal{O}_{mx}$-modules.
Set
\[
\boldsymbol{\delta^{\rm ram}}:=
\Big(\frac{1-r}{2r}\,,\,\frac{3-r}{2r}\,,\,\ldots\,,\frac{1+2i-r}{2r}
\,,\,\ldots\,,\,\frac{r-1}{2r}\Big).
\]
When $\boldsymbol{\delta} = \boldsymbol{\delta^{\rm ram}}$, 
a $(\nu(w),\boldsymbol{\delta})$-framing structure corresponds directly to a $\nu'(w)$-ramified structure.
To see this, let us first assume that a 
$(\nu(w),\boldsymbol{\delta^{\rm ram}})$-framing structure
$(\{N_i\},\{ \overline{l}_i\})$ on $(\overline{E},\overline{\nabla})$ is given.
By Lemma \ref{Lem: existence of isomorphism},
there is an isomorphism 
\begin{equation} \label {choice of isomorphism in N_0}
\sigma\colon {\mathcal A} \; 
\xrightarrow{\;\sim\;} \; \overline{E}
\end{equation}
which belongs to $N_0$.
Set
\begin{equation} \label {set of V_k,L_k from framing structure}
V_k=\sigma(w^k{\mathcal A}) \;\subset \;\overline{E},
\quad
L_k:=
w^k{\mathcal A}\big/z^{m-1}w^{k+1}{\mathcal A}
\end{equation}
for $0\leq k\leq r-1$.
Then we have
\[
\overline{E}=V_0\supset V_1\supset\cdots \supset V_{r-1}\supset V_r
=z\overline{E}
\]
and $V_k/V_{k+1}\cong w^k{\mathcal A}/w^{k+1}{\mathcal A}$
as an ${\mathcal O}_{mx}$-module whose length is $1$.
Define a surjective homomorphism
$\pi_k\colon V_k\otimes\bar{\mathcal A}\rightarrow L_k$ by
\begin{equation} \label {set of pi_k from framing structure}
 \pi_k(u\otimes \alpha(w))=\alpha(w)\overline{\sigma^{-1}(u)}
\qquad
(u\in V_k,\;\alpha(w)\in\bar{\mathcal A}),
\end{equation}
where $\bar{\mathcal A}={\mathcal A}/z^{m-1}w{\mathcal A}$.
Then the restriction 
$
\pi_k|_{V_k}\colon V_k\xrightarrow[\sim]{\sigma^{-1}}
w^k{\mathcal A}\rightarrow L_k
$
is surjective.
Let 
\begin{align} \label {set of phi_k from framing structure}
\begin{split}
&\phi_k\colon L_k\rightarrow wL_{k-1}\hookrightarrow L_{k-1}
\qquad (0\leq k\leq r-1)
\\
&\phi_r\colon (z)/(z^{m+1})\otimes L_0
=(z)/(z^{m+1})\otimes\bar{\mathcal A}
\rightarrow z\bar{\mathcal A}=wL_{r-1}
\hookrightarrow L_{r-1}
\end{split}
\end{align}
be the $\bar{\mathcal A}$-homomorphisms induced by the inclusion
$w^k{\mathcal A}\hookrightarrow w^{k-1}{\mathcal A}$
for $0\leq k\leq r$.
Then the diagrams in Definition \ref{Def: ramified structure} (iii)
are commutative.
There is a canonical isomorphism
\[
\psi'_k\colon
\big(w{\mathcal A}/z^{m-1}w^2{\mathcal A}\big)\otimes L_{k-1}
\xrightarrow{\;\sim\;} L_k
\]
of $\bar{\mathcal A}$-modules.
Its inverse $\psi_k:=(\psi'_k)^{-1}$ satisfies
the condition (iv) of  Definition \ref{Def: ramified structure}.
Since $(\{N'_i\},\{\overline{l}_i\})$ is a 
$(\nu(w),\boldsymbol{\delta^{\rm ram}})$-framing structure,
we have
\begin{equation}
\label {inclusion from framing structure}
{\mathcal D}_{\overline{\nabla},\boldsymbol{\delta^{\rm ram}}}^{(k)}(w^k\sigma)
\in \Hom_{{\mathcal O}_{mx}}({\mathcal A},\overline{l}_{k+1})
\otimes\overline{\Omega}(mx)
\end{equation}
for $0\leq k\leq r-1$.
Recall that $\nu'(w)=\nu(w)+(r-1)\mathrm{d}z/2rz$.
So
$\nu(w)^{(k)}_{\boldsymbol{\delta^{\mathrm{ram}}}}=
\nu'(w)+k\mathrm{d}z/rz$.
From \eqref{inclusion from framing structure},
$\overline{\nabla}(V_k)\subset V_k\otimes\overline{\Omega}(mx)$
holds and the diagram
\[
\begin{CD}
V_k @>\sigma^{-1}>\sim > w^k{\mathcal A} @>>>  L_k
\\
@V \overline{\nabla} VV & & @V  \nu'(w)+\frac{k\mathrm{d}z}{rz}  VV
\\
V_k\otimes\overline{\Omega}(mx) @> \sigma^{-1}\otimes\mathrm{id} > \sim > 
w^k{\mathcal A}\otimes\overline{\Omega}(mx)
@>>> L_k\otimes\overline{\Omega}(mx)
\end{CD}
\]
is commutative.
Note that the composition
$V_k\xrightarrow{\sigma^{-1}}w^k{\mathcal A}\rightarrow L_k$
is nothing but the restriction $\pi_k|_{V_k}$.
So the above diagram means the
commutativity of the diagram of (ii) in Definition \ref{Def: ramified structure}.
Thus $((V_k,L_k,\pi_k),(\phi_k))$
satisfies all the conditions of $\nu'(w)$-ramified structure
in  Definition \ref {Def: ramified structure}.

\begin{Lem} \label {Lem: ramified structure from framing structure}
For a $(\nu(w),\boldsymbol{\delta^{\rm ram}})$-framing structure
$(\{N_i\},\{\overline{l}_i\})$, the $\nu'(w)$-ramified structure
$((V_k,L_k,\pi_k),(\phi_k))$ defined by
\eqref{set of V_k,L_k from framing structure},
\eqref{set of pi_k from framing structure}
and
\eqref {set of phi_k from framing structure}
is independent of the choice of 
the isomorphism $\sigma$ in \eqref{choice of isomorphism in N_0}.
\end{Lem}

\begin{proof}
Take another isomorphism
$\sigma'\colon {\mathcal A}\xrightarrow{\sim}\overline{E}$
which belongs to $N_0$.
The equalities \eqref{set of V_k,L_k from framing structure}
and
\eqref {set of phi_k from framing structure}
defining $V_k$, $L_k$ and $\phi_k$ with respect to $\sigma'$
are the same as those with respect to $\sigma$.
Recall that we have
\[
 N_0\;\cong\;
\mathbb{C}[w]/(w^{mr})\oplus\mathbb{C}[w]/(w^{r-1})
\oplus\cdots\oplus\mathbb{C}[w]/(w)
\]
by Lemma \ref  {Lem: structure of Hom from (iv)}.
So there exists an invertible element
$\alpha(w)\in{\mathcal A}^{\times}$
such that
\[
\sigma'-\alpha(w)\sigma \: \in \:
\mathbb{C}[w]/(w^{r-1})\oplus\cdots\oplus\mathbb{C}[w]/(w).
\]
Then $w^{r-1}(\sigma'-\alpha(w)\sigma)=0$.
Let $\pi'_k\colon V_k\otimes\bar{\mathcal A}\rightarrow L_k$
be the homomorphism defined by
\[
 \pi'_k(u\otimes \beta(w))=
\beta(w)\overline{(\sigma')^{-1}(u)}
\qquad
(u\in V_k,\;\beta(w)\in\bar{\mathcal A}).
\]
Since $w^{r-1}(\sigma')^{-1}=w^{r-1}\alpha(w)^{-1}\sigma^{-1}$,
the diagram
\[
\xymatrix{
V_k \ar[rr]^{\sigma^{-1}} \ar[d]_{(\sigma')^{-1}}
&  & w^k{\mathcal A} \ar[r] 
& L_k \ar[ld]^{\alpha(w)^{-1}}
\\
w^k{\mathcal A} \ar[rr] & & L_k
}
\]
is commutative for $0\leq k\leq r-1$.
So the diagram
\[
\xymatrix{
V_k\otimes\bar{\mathcal A} \ar[r]^-{\pi_k} \ar[rd]_{\pi'_k} & 
L_k \ar[d]^{\alpha(w)^{-1}}_{\cong}
\\
& L_k
}
\]
is commutative, which means that $\pi'_k$ coincides with $\pi_k$
as a quotient $\bar{\mathcal A}$-module.
Thus $((V_k,L_k,\pi_k),(\phi_k))=((V_k,L_k,\pi'_k),(\phi_k))$.
\end{proof}

For a $\nu'(w)$-ramified structure
$((V_k,L_k,\pi_k),(\phi_k))$ on $(\overline{E},\overline{\nabla})$,
we will define a corresponding
$(\nu(w),\boldsymbol{\delta^{\rm ram}})$-framing structure.
By 
\cite[Proposition 4.1]{Inaba4} 
and its proof,
there is an isomorphism
\begin{equation} \label {isomorphism from ramified structure}
\tau\colon \overline{E}
\:\xrightarrow{\;\sim\;}\:
{\mathcal A} \,=\, \mathbb{C}[w]/(w^{rm})
\end{equation}
such that the diagram
\begin{equation} \label {commutativity from formal isomorphism}
\begin{CD}
\overline{E} @>\tau>\sim > \mathbb{C}[w]/(w^{rm})
\\
@V \overline{\nabla} VV  @VV \mathrm{d}+\nu'(w)\mathrm{id}_{\mathcal A} V
\\
\overline{E}\otimes\overline{\Omega}(mx) 
@>\tau\otimes\mathrm{id}>\sim >
\mathbb{C}[w]/(w^{rm})\otimes\frac{\mathrm{d}z}{z^m} 
\end{CD}
\end{equation}
is commutative,
$V_k=\tau^{-1}(w^k{\mathcal A})$
for $0\leq k\leq r-1$,
the induced quotient
\[
V_k\otimes_{{\mathcal O}_{mx}}\bar{\mathcal A}
\xrightarrow[\sim]{\tau\otimes\mathrm{id}}
\; w^k{\mathcal A}\otimes_{{\mathcal O}_{mx}}
\bar{\mathcal A}
\longrightarrow
\; w^k{\mathcal A}/z^{m-1}w^{k+1}{\mathcal A}
\]
coincides with $\pi_k$
and $\phi_k$ is induced by the inclusion
$w^k{\mathcal A}\hookrightarrow w^{k+1}{\mathcal A}$.
Put
\begin{equation} \label {set of l_k}
\overline{l}_k:=z^{m-1}V_k\subset \overline{E}
\end{equation}
for $0\leq k\leq r-1$.
Set
\begin{equation} \label {set of N_i}
 N_i:=
\left\{
h\in \Hom_{{\mathcal O}_{mx}}({\mathcal A},\overline{E})
\,\middle|\,
\begin{array}{l}
\text{$\exists\alpha(w)\in w^i{\mathcal A}$ such that}
\\
(h-\tau^{-1}\circ\alpha(w)\mathrm{id}_{\mathcal A})(w^k{\mathcal A})
\subset \overline{l}_{k}
\\
\text{for $0\leq k\leq r-1$}
\end{array}
\right\}
\end{equation}
for $0\leq i\leq rm-r$.
Then we have
\[
\length N_i
=
\length (w^i{\mathcal A})
+\frac{r(r+1)}{2}-\length (w^{rm-r}{\mathcal A})
=
rm-i+\frac{r(r-1)}{2}
\]
for $0\leq i\leq rm-r$.
Set
\begin{equation} \label {set of N_{rm-r+j}}
N_{rm-r+j}
:=
\left\{
h\in \Hom_{{\mathcal O}_{mx}}({\mathcal A},\overline{E})
\,\middle|\,
\begin{array}{l}
\text{$w^{r-j}h=0$ and $h(w^k{\mathcal A})\subset \overline{l}_k$}
\\
\text{for $0\leq k\leq r-1$}
\end{array}
\right\}
\end{equation}
for $0\leq j\leq r$.
Then we have
\[
\length N_{rm-r+j}
=
\frac{r(r+1)}{2}-\frac{j(j+1)}{2}
\]
for $0\leq j\leq r$.
So $(\{N_i\},\{\overline{l}_i\})$ satisfies
the conditions (i) and (ii) of Definition~\ref {def: framing structure}.
By the commutativity of \eqref {commutativity from formal isomorphism},
we have
\begin{align*}
{\mathcal D}^{(k)}_{\overline{\nabla},\boldsymbol{\delta^{\rm ram}}}
(w^k\tau^{-1})(f(w))
&=
\overline{\nabla}(\tau^{-1}(w^kf(w)))
-\tau^{-1}(\nu(w)^{(k)}_{\boldsymbol{\delta^{\rm ram}}}(w^kf(w)))
\\
&\equiv 0  \pmod{\overline{l}_{k+1}\otimes\overline{\Omega}(mx)}
\end{align*}
for any $f(w)\in{\mathcal A}$.
Recall that
$N_0={\mathcal A} \tau^{-1} +N_{rm-r}$.
For any $h\in N_{rm-r}$,
there is $c\in\mathbb{C}$ such that
$h':=h-cz^{m-1}\tau^{-1}$ satisfies
$h'(w^k)\subset \overline{l}_{k+1}$.
So 
\[
{\mathcal D}^{(k)}_{\overline{\nabla},\boldsymbol{\delta^{\rm ram}}}
(w^kN_0) \subset
\Hom({\mathcal A},\overline{l}_{k+1})\otimes\overline{\Omega}(mx)
\]
holds for $0\leq k\leq r-1$.
Moreover, we have
$\tau^{-1}\in N_0$ and $w^{rm-1}\tau^{-1}\neq 0$.
Thus $(\{N_i\},\{\overline{l}_i\})$
is a $(\nu(w),\boldsymbol{\delta^{\rm ram}})$-framing structure.

\begin{Lem} \label {Lem: framing structure from ramified structure}
For a $\nu'(w)$-ramified structure $((V_k,L_k,\pi_k),(\phi_k))$,
the $(\nu(w),\boldsymbol{\delta^{\rm ram}})$-framing structure
$(\{N_i\},\{\overline{l}_i\})$ defined by
\eqref {set of l_k}
\eqref {set of N_i} and \eqref{set of N_{rm-r+j}}
is independent of the choice of the isomorphism
$\tau$ in \eqref {isomorphism from ramified structure}.
\end{Lem}

\begin{proof}
Take another isomorphism
$\tau'\colon E|_{mx}\xrightarrow{\sim}{\mathcal A}$
with the same property as $\tau$.
Then there exist $\overline{\alpha}(w)\in\mathbb{C}[w]/(z^{m-1}w)$
and $\beta_0,\ldots,\beta_{r-1}\in\mathbb{C}$
such that the diagram
\[
\xymatrix{
V_k \ar[rr]^-{\tau|_{V_k}} \ar[rrd]_{\tau'|_{V_k}}
& & w^k{\mathcal A} \ar[r] & w^k{\mathcal A}/w^{k+rm-r+1}{\mathcal A}
\ar[d]^{\overline{\alpha}(w)+\beta_kz^{m-1}}
\\
&  & w^k{\mathcal A} \ar[r] & w^k{\mathcal A}/w^{k+rm-r+1}{\mathcal A}
}
\]
is commutative for $0\leq k\leq r-1$.
Take a lift $\alpha(w)\in{\mathcal A}$ of $\overline{\alpha}(w)$.
For $0\leq k\leq r-1$, we have
$(\alpha(w)\mathrm{id}_{\mathcal A}\circ\tau-\tau')
(V_k)\subset z^{m-1}w^{k}{\mathcal A}$,
which means
$(\tau^{-1}\circ\alpha(w)^{-1}\mathrm{id}_{\mathcal A}-\tau'^{-1})
(w^k)\subset \overline{l}_k$.
So the submodules 
$N_i\subset
\Hom_{{\mathcal O}_{mx}}({\mathcal A},\overline{E})$
defined by \eqref{set of N_i} and \eqref{set of N_{rm-r+j}}
regarding $\tau$ are the same as those regarding $\tau'$.
\end{proof}

\begin{Prop} \label {Prop: (nu,n)-framing structure and ramified structure}
There is a bijective correspondence between the set of
$(\nu(w),\boldsymbol{\delta^{\rm ram}})$-framing structures
on $(\overline{E},\overline{\nabla})$ and the set of 
$\nu'(w)$-ramified structures on $(\overline{E},\overline{\nabla})$.
\end{Prop}

\begin{proof}
By Lemma \ref{Lem: ramified structure from framing structure},
we get a correspondence from
the $(\nu(w),\boldsymbol{\delta^{\rm ram}})$-framing structures
$(\{N_i\},\{\overline{l}_i\})$ to the $\nu'(w)$-ramified structures
$((V_k,L_k,\pi_k),(\phi_k))$.
Lemma \ref{Lem: framing structure from ramified structure}
provides the opposite correspondence
and we can see that these are inverse to each other.
\end{proof}

By 
Proposition \ref {Prop: (nu,c)-framing structure and (nu,n)-framing structure}
and
Proposition \ref {Prop: (nu,n)-framing structure and ramified structure},
we get the following corollary.

\begin{Cor} \label {Cor: framing structure and ramified structure}
Let $\boldsymbol{\delta}=(\delta_0,\ldots,\delta_{r-1})$
be a tuple of complex numbers satisfying \eqref {trace condition}.
Then there is a bijective correspondence between the set of
$(\nu(w),\boldsymbol{\delta})$-framing structures on 
$(\overline{E},\overline{\nabla})$
and the set of $\nu'(w)$-ramified structures on 
$(\overline{E},\overline{\nabla})$.
\end{Cor}

\subsection{$(\nu,{\mathcal A})$-frame} \label {subsect nu A frame}

In view of Corollary \ref{Cor: framing structure and ramified structure},
we may choose $\boldsymbol{\delta}$ to be
$\boldsymbol{0}=(0,0,\ldots,0)$.
Then we can replace the notion of $(\nu(w),\boldsymbol{0})$-framing structure
with a simpler notion.

\begin{Def} \label {def: (nu,A)-frame}
Let $(\overline{E},\overline{\nabla})$ be a pair consisting of
a free ${\mathcal O}_{mx}$-module $\overline{E}$ of rank $r$
and a homomorphism
$\overline{\nabla}\colon\overline{E}\rightarrow\overline{E}\otimes\overline{\Omega}(mx)$
of $\mathcal{O}_{mx}$-modules.
We call an ${\mathcal A}$-submodule
$L\subset\Hom_{{\mathcal O}_{mx}}({\mathcal A},\overline{E})$
a \emph{$(\nu(w),{\mathcal A})$-frame} on $(\overline{E},\overline{\nabla})$ if
\begin{itemize}
\item[\textup{(i)}]
$L$ is a free $\mathcal{A}$-module of rank one,
\item[\textup{(ii)}]
a generator $\sigma$ of $L$ is an isomorphism
$\sigma\colon{\mathcal A}\xrightarrow{\sim}\overline{E}$
of ${\mathcal O}_{mx}$-modules,
and
\item[\textup{(iii)}]
the equality
$(\sigma^{-1}\otimes\mathrm{id}_{\overline{\Omega}(mx)})
\circ\overline{\nabla}\circ\sigma
=\nu(w)\mathrm{id}_{\mathcal A}$
holds.
\end{itemize}
\end{Def}

\begin{Prop} \label {Prop: rephrase to frame}
There is a bijection between the set of
$(\nu(w),\boldsymbol{0})$-framing structures on $(\overline{E},\overline{\nabla})$
and the set of $(\nu(w),{\mathcal A})$-frames on
$(\overline{E},\overline{\nabla})$.
\end{Prop}

\begin{proof}
Given a $\nu(w)$-framing structure
$(\{N_j\},\{\overline{l}_k\})$ on $(\overline{E},\overline{\nabla})$.
By Lemma \ref {Lem: existence of isomorphism},
there is an isomorphism
$\sigma\colon {\mathcal A}\xrightarrow{\sim}\overline{E}$
which belongs to $N_0$.
Then the representation matrix of 
$\overline{\nabla}$ is of the form
given in Proposition \ref{natural representation matrix},
with $\delta_0=\cdots=\delta_{r-1}=0$.
Replacing $\sigma$ with an addition of $N_{rm-r}$,
we may further assume that the representation matrix of $\overline{\nabla}$
in Proposition \ref{natural representation matrix} satisfies
$b_{ij}=0$ for $i>j$.
Then $\sigma$ satisfies the equality of Definition \ref{def: (nu,A)-frame} (iii).
So 
${\mathcal A}\sigma\subset
\Hom_{{\mathcal O}_{mx}}({\mathcal A},\overline{E})$
becomes a $(\nu(w),{\mathcal A})$-frame on 
$(\overline{E},\overline{\nabla})$.
Take another isomorphism
$\sigma'\colon {\mathcal A}\xrightarrow{\sim}\overline{E}$
such that 
${\mathcal A} \sigma'\subset
\Hom_{{\mathcal O}_{mx}}({\mathcal A},\overline{E})$
is a $(\nu(w),{\mathcal A})$-frame of $(\overline{E},\overline{\nabla})$.
Then the composition 
$\sigma^{-1}\circ \sigma'\colon{\mathcal A}\rightarrow{\mathcal A}$
commutes with $\nu(w)\mathrm{id}_{\mathcal A}$, which implies
the existence of $\alpha(w)\in{\mathcal A}^{\times}$ satisfying
$\sigma^{-1}\circ \sigma'=\alpha(w)\mathrm{id}_{\mathcal A}$.
So the corresponding submodule
${\mathcal A}\sigma={\mathcal A}\sigma'$
of $\Hom_{{\mathcal O}_{mx}}({\mathcal A},\overline{E})$
is independent of the choice of the isomorphism $\sigma$.

Conversely assume that a $(\nu(w),{\mathcal A})$-frame 
$L\subset\Hom_{{\mathcal O}_{mx}}({\mathcal A},\overline{E})$
on $(\overline{E},\overline{\nabla})$ is given.
Take a generator $\sigma$ of $L$, which provides
an isomorphism $\sigma\colon {\mathcal A}\xrightarrow{\sim}\overline{E}$
satisfying the equality of Definition \ref{def: (nu,A)-frame} (iii).
Define $\{\overline{l}_k\}$ by
$\overline{l}_k=z^{m-1}\sigma(w^k{\mathcal A})$
and define $\{N_j\}$ by
\[
N_j:=
\left\{
h\in \Hom_{{\mathcal O}_{mx}}
({\mathcal A}, \overline{E} )
\,\middle|\,
\begin{array}{l}
\exists \alpha(w)\in w^j{\mathcal A},
\;
\text{for $0\leq k\leq r-1$,}
\\
(h-\alpha(w)\sigma)(w^k{\mathcal A})
\subset \overline{l}_k
\end{array}
\right\}
\]
for $0\leq j\leq rm-r$
and
\[
N_{rm-r+j}
=
\left\{
h\in \Hom_{{\mathcal O}_{mx}}
({\mathcal A}, \overline{E})
\,\middle|\,
\begin{array}{l}
\text{$w^{r-j}h=0$ and $h(w^k{\mathcal A})
\subset \overline{l}_k$}
\\
\text{for $0\leq k\leq r-1$}
\end{array}
\right\}
\]
for $0\leq j\leq r$.
Then $(\{N_j\},\{\overline{l}_k\})$ becomes a 
$(\nu(w),\boldsymbol{0})$-framing structure.
\end{proof}

\section{Moduli space of $\lambda$-connections} \label{sect: Moduli space}

The purpose of this section is 
to provide the algebraic construction of 
the moduli space of $\lambda$-connections admitting ramified singularities 
(Theorem~\ref{thm: main_moduli_construction}).

Let $X$ be a smooth projective irreducible
algebraic curve over $\mathbb{C}$ of genus $g$.
Take an effective divisor $D$ on $X$ and write
$ D=\sum_{x\in D} m_xx$.
At each point $x\in D$,
choose a generator $z$ of the maximal ideal of ${\mathcal O}_{X,x}$.
Let $r^x_1,\ldots,r^x_{s_x}$ be positive integers
satisfying $r=r^x_1+\cdots+r^x_{s_x}$.
Set
${\mathcal A}^x_i:=
{\mathcal O}_{m_xx}[w]/(w^{r^x_i}-z)$
for $1\leq i\leq s_x$.
As we saw before Definition~\ref{def parabolic frame in Intro},
the space of local exponents $\Lambda$ is defined by
\[
 \Lambda
 :=
 \left\{
 (\lambda,\nu)
 \,\middle|\,
 \begin{array}{l}
 \lambda\in\mathbb{C}, \;
 \nu \text{ is a generic ramified local exponent data for $\lambda$} \\
 \text{defined in Definition
 \ref{def: generic ramified local exponent data intro}}
 \end{array}
 \right\}.
\]
There is a canonical morphism
$
\Lambda \rightarrow \mathbb{A}^1
$
which sends $(\lambda,\nu)$ to $\lambda$.
For $x\in D$ and  $1\leq i\leq s_x$, let
$\boldsymbol{\delta}^x_i=(\delta^x_{i,0},\ldots,\delta^x_{i,r^x_i-1})$
be a tuple of complex numbers satisfying
$\delta^x_{i,0}+\cdots+\delta^x_{i,r^x_i-1}=0$.
Put
$\boldsymbol{\delta}:=(\boldsymbol{\delta}^x_i)^{x\in D}_{1\leq i\leq s_x}$.

\begin{Def} \label {def: parabolic framing structure}
Take $(\nu,\lambda)\in\Lambda$
and write $\nu = (\nu^{x,i}(w))_{x \in D, \, 1 \leq i \leq s_x}$.
Let $E$ be a vector bundle on $X$ of rank $r$ and degree $d$,
and let $\nabla \colon E \rightarrow E \otimes \Omega^1_X(D)$ be a $\lambda$-connection on $E$.
We say that $(\{l^x_i, N^{x,i}_j, \overline{l}^{x,i}_k\})$
is a \emph{parabolic $(\nu,\boldsymbol{\delta})$-framing structure 
on $(E,\nabla)$} if
\begin{itemize}
\item[\textup{(i)}]
$E|_{m_xx}=l^x_1\supset\cdots\supset l^x_{s_x}\supset l^x_{s_x+1}=0$
is a filtration by ${\mathcal O}_{m_xx}$-submodules
such that $l^x_i/l^x_{i+1}\cong{\mathcal O}_{m_xx}^{\oplus r^x_i}$
and $\nabla|_{m_xx}(l^x_i)\subset l^x_i\otimes\Omega^1_X(D)|_{m_xx}$
for $1\leq i\leq s_x$, and
\item[\textup{(ii)}]
for $V^x_i:=l^x_i/l^x_{i+1}$ and for the homomorphism
$\overline{\nabla}^x_i\colon V^x_i\rightarrow
V^x_i\otimes\Omega^1_X(D)|_{m_xx}$
induced by $\nabla|_{m_xx}$,
a tuple $\{N^{x,i}_j, \overline{l}^{x,i}_k\}$
is a $(\nu^{x,i}(w),\boldsymbol{\delta}^x_i)$-framing structure
on $(V^x_i,\overline{\nabla}^x_i)$
in the sense of Definition \ref{def: framing structure}.
\end{itemize}
\end{Def}

\begin{Rem}
Proposition~\ref{Prop: (nu,c)-framing structure and (nu,n)-framing structure} yields that
a parabolic $(\nu,\boldsymbol{\delta})$-framing structure 
is equivalent to a $(\nu,\boldsymbol{0})$-framing structure,
where $\boldsymbol{0}=(\boldsymbol{0}^x_i)^{x\in D}_{1\leq i\leq s_x}$
and $\boldsymbol{0}^x_i=(0,\ldots,0)$.
So one may assume $\boldsymbol{\delta}=\boldsymbol{0}$
for the construction of the moduli space of $\lambda$-connections
with a parabolic $(\nu,\boldsymbol{\delta})$-framing structure.
\end{Rem}

Proposition \ref{Prop: rephrase to frame} yields the following.

\begin{Prop} \label {Prop: framing-frame correspondence}
Let $E$ be a vector bundle on $X$ 
and let $\nabla\colon E\rightarrow E\otimes\Omega^1_X(D)$
be a $\lambda$-connection.
Then there is a bijective correspondence between the set of
parabolic $(\nu,{\mathcal A})$-frames on $(E,\nabla)$
and the set of parabolic $(\nu,\boldsymbol{0})$-framing structures
on $(E,\nabla)$.
\end{Prop}

For a parabolic $(\nu,\boldsymbol{0})$-framing structure
$\{l^x_i,N^{x,i}_j,\overline{l}^{x,i}_k\}$
on a $\lambda$-connection 
$(E,\nabla)$, define
\begin{equation}\label{definition of extended filtration}
\tilde{l}^{x,i}_0:=l^x_i, \qquad
\tilde{l}^{x,i}_k:=
\ker\big(l^x_i\rightarrow (l^x_i/l^x_{i+1})/\overline{l}^{x,i}_k\big)
\quad (1\leq k\leq r^x_i)
\end{equation}
for $x\in D$ and $1\leq i\leq s_x$.

\begin{Def} \label {def: stability of lambda connection}
Let $(E,\nabla,\{l^x_i,N^{x,i}_j,\overline{l}^{x,i}_k\})$
be a  $\lambda$-connection 
with a parabolic $(\nu,\boldsymbol{0})$-framing structure.
Take the filtration $\{\tilde{l}^{x,i}_k\}$ on $E|_{m_xx}$
defined by \eqref{definition of extended filtration}.
Let 
$\alpha=(\alpha^{x,i}_k)^{x\in D}_{1\leq i\leq s_x,0\leq k\leq r^x_i-1}$
be a tuple of positive rational numbers such that
\[
0<\alpha^{x,1}_0<\cdots<\alpha^{x,1}_{r^x_1-1}<
\alpha^{x,2}_0<\alpha^{x,2}_1<\cdots<\alpha^{x,2}_{r^x_2-1}<\cdots
<\alpha^{x,s_x}_0
<\cdots<\alpha^{x,s_x}_{r^x_{s_x}-1}<1
\]
for $x\in D$ and
$\alpha^{x,i}_k\neq \alpha^{x',i'}_{k'}$
for $(x,i,k)\neq(x',i',k')$.
We say that 
$(E,\nabla,\{l^x_i,N^{x,i}_j,\overline{l}^{x,i}_k\})$
is \emph{$\alpha$-stable} if
for any nontrivial subbundle $0\neq F\subsetneq E$
satisfying $\nabla(F)\subset F\otimes\Omega^1_X(D)$,
the inequality
\begin{align*}
 &\frac{\deg F+\sum_{x\in D}\sum_{i=1}^{s_x}\sum_{k=0}^{r^x_i-1}
 \alpha^{x,i}_k
 \length\big((F|_{m_xx}\cap \tilde{l}^{x,i}_k)/(F|_{m_xx}\cap \tilde{l}^{x,i}_{k+1})\big)}
 {\rank F}
\\
&<
\frac{\deg E+\sum_{x\in D}\sum_{i=1}^{s_x}\sum_{k=0}^{r^x_i-1}
 \alpha^{x,i}_k
 \length\big(\tilde{l}^{x,i}_k/\tilde{l}^{x,i}_{k+1}\big)} {\rank E}
\end{align*}
holds.
We say that $(E,\nabla,\{l^x_i,N^{x,i}_j,\overline{l}^{x,i}_k\})$
is \emph{$\alpha$-semistable} if
the weaker inequality ``$\,\leq$'' instead of ``$\,<$'' holds.
\end{Def}

\begin{Def} \label {def: moduli functor: parabolic framing structure}
Denote the category of sets by $(\mathrm{Sets})$
and the category of noetherian schemes over $\Lambda$ by
$(\mathrm{Sch}/\Lambda)$.
Define a functor
$
 {\mathcal M}^{\alpha}_{\mathrm{Hod}}
 \colon (\mathrm{Sch}/\Lambda) \longrightarrow(\mathrm{Sets})
$
by
\[
{\mathcal M}^{\alpha}_{\mathrm{Hod}}(T)
=
\left\{
(E,\nabla,\{l^x_i,N^{x,i}_j,\overline{l}^{x,i}_k\})
\,\middle|\,
\begin{array}{l}
\text{flat family of $\alpha$-stable $\lambda$-connections}
\\
\text{with a parabolic $(\nu,\boldsymbol{0})$-framing structure}
\\
\text{on $X\times T$ over $T$}
\end{array}
\right\}\Big/\!\sim,
\]
for $T\in(\mathrm{Sch}/\Lambda)$ and
$(\nu,\lambda)\in\Lambda(T)$ is determined by
the morphism $T\rightarrow\Lambda$.
Here we define
$(E,\nabla,\{l^x_i,N^{x,i}_j,\overline{l}^{x,i}_k\})\sim
(E',\nabla',\{(l')^x_i,(N')^{x,i}_j,\overline{(l')}^{x,i}_k\})$
if there is a line bundle ${\mathcal L}$ on $T$ satisfying
\[
(E,\nabla,\{l^x_i,N^{x,i}_j,\overline{l}^{x,i}_k\})
\cong
(E',\nabla',\{(l')^x_i, (N')^{x,i}_j, \overline{(l')}^{x,i}_k\})
\otimes{\mathcal L}.
\]
\end{Def}

According to Proposition \ref{Prop: framing-frame correspondence},
Theorem~\ref{thm: main_moduli_construction} follows from the following theorem.

\begin{Thm}  \label {Thm: existence of moduli space}
There exists a coarse moduli scheme
$M^{\alpha}_{\mathrm{Hod}}$ of the moduli functor
${\mathcal M}^{\alpha}_{\mathrm{Hod}}$
defined in Definition~\ref{def: moduli functor: parabolic framing structure}.
Furthermore, $M^{\alpha}_{\mathrm{Hod}}$
is a quasi-projective scheme over $\Lambda$
and it represents the \'etale sheafification
of the moduli functor
${\mathcal M}^{\alpha}_{\mathrm{Hod}}$.
\end{Thm}

\begin{Rem} \label {rem: formal data of parabolic framing structure} 
For 
$(E,\nabla,\{l^x_i,N^{x,i}_j,\overline{l}^{x,i}_k\})
\in M^{\alpha}_{\mathrm{Hod}}$
and a point $x \in D$ with $m_x \geq 2$,
its formal completion at $x$ has a decomposition
\[
 (\widehat{E}_x,\widehat{\nabla}_x)
 \cong
 \bigoplus_{i=1}^{s_x} (\widehat{E}^x_i,\widehat{\nabla}^x_i)
\]
such that $(\widehat{E}^x_i,\widehat{\nabla}^x_i)|_{m_x x}\cong
(l^x_i/l^x_{i+1},\overline{\nabla}^x_i)$.
If $\lambda=1$, we have
\[
(\widehat{E}^x_i,\widehat{\nabla}^x_i)
\cong
(\mathbb{C}[[w]],\mathrm{d}+\nu'^{x,i}(w)), 
\]
where $\nu'^{x,i}(w):=\nu^{x,i}(w)-\frac{(r^x_i-1){\rm d} z}{2r^x_iz}$.
\end{Rem}

\begin{Rem} 
When $D \neq \emptyset$, if $m_x \geq 2$ for some $x \in D$, we set
\[
E' := \ker \!\big( E \rightarrow \widehat{E}^x_i|_{m_x x}/\overline{l}^{x,i}_k \big)
\]
for a suitable choice of indices $i$ and $k$, using the formal decomposition in 
Remark~\ref{rem: formal data of parabolic framing structure}.
Then we obtain a $\lambda$-connection $(E',\nabla')$ equipped 
with a canonical parabolic framing structure such that $\operatorname{rank} E'$ 
and $\operatorname{deg} E'$ are coprime.
Thus, an argument similar to the proof of \cite[Theorem 2.1]{Inaba1} 
yields that $M^{\alpha}_{\mathrm{Hod}}$ is a fine moduli scheme, when $D \neq \emptyset$.
\end{Rem}

\subsection{Construction of the moduli space}\label{subsect: Moduli space 2}

The aim of this section is to prove 
Theorem \ref{Thm: existence of moduli space}.
We use the method of \cite[Section 5]{Inaba-Iwasaki-Saito1}
with a slight modification.

Let $X\rightarrow S$ be a smooth projective morphism
of noetherian schemes such that
each geometric fiber is an irreducible curve of genus $g$. 
Assume that a morphism $S\rightarrow\mathbb{A}^1$ is given.
Let $D\subset X$ be a relative effective Cartier divisor over $S$.
Fix a relatively ample line bundle ${\mathcal O}_X(1)$ on $X$ over $S$.
Set
\[
\tilde{\Lambda}^1_{D/S}:=
{\mathcal O}_{X}
\oplus
\big(\Omega^1_{X/S}
(D_{S})\big)^{\vee}
\]
and define a structure of ${\mathcal O}_{X}$-bimodule by
\[
 (a,v)f=(fa+t\langle \mathrm{d}f,v\rangle,fv), \quad
 f(a,v)=(fa,fv)
\]
for $f,a\in{\mathcal O}_{X}$,
$v\in\big(\Omega^1_{X/S}(D)\big)^{\vee}$,
where $t$ is the canonical coordinate of $\mathbb{A}^1$.
Fix a geometric point $\lambda$ of $\mathbb{A}^1$
and a geometric point $s$ of $S$ over $\lambda$.
Denote the fiber of $X,D$ over $s$ by $X_s,D_s$, respectively.
Denote the restriction
$\tilde{\Lambda}^1_{D/S}|_{X_s\times\{\lambda\}}$ by
$\Lambda^1_{D_s,\lambda}$.
Then giving a $\lambda$-connection
$\nabla\colon E\rightarrow E\otimes\Omega^1_{X_s}(D_s)$
on a vector bundle $E$ on $X_s$ is equivalent to giving a homomorphism
$\Phi\colon \Lambda^1_{D_s,\lambda}\otimes E\rightarrow E$
of left ${\mathcal O}_{X_s}$-modules
such that the composition
$
E={\mathcal O}_{X_s}\otimes E\hookrightarrow
\Lambda^1_{D_s,\lambda}\otimes E\rightarrow E
$
coincides with the identity $\mathrm{id}_E$.

\begin{Def}
We say that
$(E_1,E_2,\Phi,F_*(E_1))$ is a \emph{parabolic $\Lambda^1_{D_s,\lambda}$-tuple} if
\begin{itemize}
\item[{\rm (i)}] $E_1,E_2$ are vector bundles on $X_s$,
\item[{\rm (ii)}]
$\Phi\colon \Lambda^1_{D_s,\lambda}\otimes E_1\rightarrow E_2$
is a homomorphism of left ${\mathcal O}_{X_s}$-modules,
\item[{\rm (iii)}]
$E_1=F_1(E_1)\supset F_2(E_1)\supset\cdots\supset
F_l(E_1)\supset F_{l+1}(E_1)=E_1(-D_s)$
is a filtration by coherent subsheaves.
\end{itemize}
\end{Def}

\begin{Rem} 
Let $(E_1,E_2,\Phi,F_*(E_1))$ be a $\Lambda^1_{D_s,\lambda}$-tuple
such that the composition
\[
\phi\colon E_1={\mathcal O}_{X_s}\otimes E_1
\hookrightarrow \Lambda^1_{D_s,\lambda}\otimes E_1
\xrightarrow{\;\Phi\;} E_2
\]
is an isomorphism. 
Then 
$\phi^{-1}\circ\Phi\colon
\Lambda^1_{D_s,\lambda}\otimes E_1\rightarrow E_1$
corresponds to a $\lambda$-connection.
Since $(E_1,E_2,\Phi,F_*(E_1))\cong(E_1,E_1,\phi^{-1}\circ\Phi,F_*(E_1))$
holds, $(E_1,E_2,\Phi,F_*(E_1))$ can be regarded as a $\lambda$-connection.
\end{Rem}

\begin{Def}
Let $\alpha'=(\alpha_i)_{1\leq i\leq l+1}$ be a tuple of rational numbers satisfying
$0\leq \alpha'_1<\alpha'_2<\cdots<\alpha'_l<\alpha'_{l+1}=1$,
$\beta=(\beta_1,\beta_2)$ a pair of positive integers
and $\gamma> 0$ a large positive rational number.
\begin{itemize}
\item[{\rm (1)}]
For a parabolic $\Lambda^1_{D_s,\lambda}$-tuple
$(E_1,E_2,\Phi,F_*(E_1))$,
set
\begin{align*}
\mathrm{pardeg}(E_1,E_2,\Phi,F_*(E_1))
&:=
\beta_1\deg E_1(-D_s)+\beta_2(\deg E_2
-\gamma\deg{\mathcal O}_{X_s}(1)\rank E_2)
\\
&\qquad
+\sum_{i=1}^l\beta_1\alpha'_i\length (F_i(E_1)/F_{i+1}(E_1)).
\end{align*}
\item[{\rm (2)}]
We say that a parabolic $\Lambda^1_{D_s,\lambda}$-tuple
$(E_1,E_2,\Phi,F_*(E_1))$ is
\emph{$(\alpha',\beta,\gamma)$-stable} if
\[
\frac{\mathrm{pardeg}(E'_1,E'_2,\Phi',F_*(E'_1)) }
{\rank E'_1+\rank E'_2}
<\frac{\mathrm{pardeg}(E_1,E_2,\Phi,F_*(E_1))}
{\rank E_1+\rank E_2}
\]
holds for any sub-tuple
$(E'_1,E'_2,\Phi',F_*(E'_1))\subset(E_1,E_2,\Phi,F_*(E_1))$,
which means that
$E'_1\subset E_1$, $E'_2\subset E_2$ are coherent subsheaves,
$\Phi'\colon \Lambda^1_{D_s,\lambda}\otimes E'_1\rightarrow E'_2$
is a homomorphism of ${\mathcal O}_{X_s}$-modules satisfying
$\Phi|_{\Lambda^1_{D_s,\lambda}\otimes E'_1}=\Phi'$
and $F_*(E'_1)$ is a filtration of
$E'_1$ satisfying $F_i(E'_1)\subset F_i(E_1)$ for any $i$.
We say $(E_1,E_2,\Phi,F_*(E_1))$ is
\emph{$(\alpha',\beta,\gamma)$-semistable} if the weaker inequality
`$\leq$' holds instead of `$<$'.
\end{itemize}
\end{Def}

\begin{Def}
Denote the category of sets by $(\mathrm{Sets})$
and the category of noetherian schemes over $S$ by
$(\mathrm{Sch}/S)$.
Define a functor
$
 \overline{\mathcal M}^{\alpha',\beta,\gamma}_{\tilde{\Lambda}^1_{D/S}}(r,(d_i))
 \colon (\mathrm{Sch}/S) \longrightarrow(\mathrm{Sets})
$
by
\[
\overline{\mathcal M}^{\alpha',\beta,\gamma}_{\tilde{\Lambda}^1_{D/S}}(r,(d_i))
(T)
=
\left\{
(E_1,E_2,\Phi,F_*(E_1))\,\middle|\,
\begin{array}{l}
\text{\rm 
$(E_1,E_2,\Phi,F_*(E_1))$ satisfies}
\\
\text{\rm the following (i),(ii),(iii)}
\end{array}
\right\}\Big/\!\sim,
\]
for $T\in(\mathrm{Sch}/S)$, where
\begin{itemize}
\item[\textup{(i)}]
$(E_1,E_2,\Phi,F_*(E_1))$ is a flat family of
parabolic $\tilde{\Lambda}^1_{D_T/T}$-tuples on $X_T$ over $T$,
\item[\textup{(ii)}]
for any $t\in T$,
$\rank E_1|_{X_t}=\rank E_2|_{X_t}=r$,
$\deg E_1|_{X_t}=\deg E_2|_{X_t}=d_1$ and
$\deg F_i(E_1)=d_i$ for $1\leq i\leq l$,
\item[\textup{(iii)}]
$(E_1,E_2,\Phi,F_*(E_1))|_{X_t}$ is
$(\alpha',\beta,\gamma)$-stable
for any geometric point $t$ of $T$.
\end{itemize}
Here $(E_1,E_2,\Phi,F_*(E_1))\sim(E'_1,E'_2,\Phi,F_*(E'_1))$
if there are a line bundle ${\mathcal L}$ on $T$ and
isomorphisms
$\theta_1\colon E_1\xrightarrow{\sim}E'_1\otimes{\mathcal L}$,
$\theta_2\colon E_2\xrightarrow{\sim}E'_2\otimes{\mathcal L}$
such that
\[
\begin{CD}
\tilde{\Lambda}^1_{D_T/T}\otimes E_1 @>\theta_1>\sim > 
\tilde{\Lambda}^1_{D_T/T}\otimes E'_1\otimes{\mathcal L}
\\
@V \Phi VV @VV \Phi' \otimes \mathrm{id}_{\mathcal{L}} V
\\
E_2 @>\theta_2>\sim>  E'_2\otimes{\mathcal L}
\end{CD}
\]
is commutative and
$\theta_1(F_i(E_1))=F_i(E'_1)\otimes{\mathcal L}$
for $1\leq i\leq l$.
\end{Def}

\begin{Thm}[{\cite[Theorem 5.1]{Inaba-Iwasaki-Saito1}}]
 \label {Thm: moduli of Lambda-tuples}
There exists a relative coarse moduli scheme 
\[
\overline{M}^{\alpha',\beta,\gamma}_{\tilde{\Lambda}^1_{D/S}}(r,(d_i))
\longrightarrow S
\]
of  $\overline{\mathcal M}^{\alpha',\beta,\gamma}_{\tilde{\Lambda}^1_{D/S}}(r,(d_i))$.
In fact, 
$\overline{M}^{\alpha',\beta,\gamma}_{\tilde{\Lambda}^1_{D/S}}(r,(d_i))$
represents the \'etale sheafification of the moduli functor
$\overline{\mathcal M}^{\alpha',\beta,\gamma}_{\tilde{\Lambda}^1_{D/S}}(r,(d_i))$.
If $\alpha'$ is generic so that every
$(\alpha',\beta,\gamma)$-semistable parabolic tuple is
$(\alpha',\beta,\gamma)$-stable,
then $\overline{M}^{\alpha',\beta,\gamma}_{\tilde{\Lambda}^1_{D/S}}(r,(d_i))$
is projective over $S$.
\end{Thm}

\begin{proof}
The proof is exactly the same as that of 
\cite[Theorem 5.1 and Proposition 5.6]{Inaba-Iwasaki-Saito1}
by  replacing $\Lambda^1_D$ with
$\tilde{\Lambda}^1_{D/S}$.
\end{proof}

\noindent
{\it Proof of Theorem \ref{Thm: existence of moduli space}.}
We use the same argument as \cite[Section 2.4]{Inaba-Iwasaki-Saito1}.
Let $\tilde{\alpha}_1,\ldots,\tilde{\alpha}_l$ be the reordering of
$\{\alpha^{x,i}_k\}$ such that
$0<\tilde{\alpha}_1<\cdots<\tilde{\alpha}_l<1$.
Choose positive integers $\beta_1,\beta_2$ 
so that $\alpha'_i:=\frac{\beta_1+\beta_2}{\beta_1}\tilde{\alpha}_i$
satisfies $0<\alpha'_i<1$ for any $i$.
Moreover we take a sufficiently large rational number $\gamma$.
For  a $\lambda$-connection with a parabolic
$(\nu,\boldsymbol{0})$-framing structure
$(E,\nabla,\{l^x_i,N^{x,i}_j,\overline{l}^{x,i}_k\})$
on a geometric fiber $X_s$,
let  $\{\tilde{l}^{x,i}_k\}$ be the filtration on $E|_{m_xx}$ 
defined by \eqref{definition of extended filtration}.
Let $\Phi\colon \Lambda^1_{D_s,\lambda}\otimes E\rightarrow E$
be the homomorphism of left ${\mathcal O}_{X_s}$-modules
determined by $(\lambda,\nabla)$.
Moreover, set $F_1(E):=E$ and 
$F_j(E):=\ker (F_{j-1}(E)\hookrightarrow E\rightarrow E|_{m_xx}/\tilde{l}^{x,i}_k)$
where $j$ is determined by
$\tilde{\alpha}_j=\alpha^{x,i}_k$.
Put $d_j:=\deg F_j(E)$.
Then $(E,\nabla,\{l^x_i,N^{x,i}_j,\overline{l}^{x,i}_k\})$
is $\alpha$-stable
if and only if the associated parabolic $\Lambda^1_{D_s,\lambda}$-tuple
$(E,E,\Phi,F_*(E))$ is $(\alpha',\beta,\gamma)$-stable.
The correspondence
$(E,\nabla,\{l^x_i,N^{x,i}_j,\overline{l}^{x,i}_k\})
\mapsto (E,E,\Phi,F_*(E))$
can be extended to a morphism of functors
\begin{equation} \label {morphism to Lambda tuple functor}
 {\mathcal M}^{\alpha}_{\mathrm{Hod}}
 \;\longrightarrow\;
 \overline{\mathcal M}^{\alpha',\beta,\gamma}
 _{\tilde{\Lambda}^1_{D/\Lambda}}(r,(d_i)).
\end{equation}
Recall that the moduli space
$\overline{M}^{\alpha',\beta,\gamma}_{\tilde{\Lambda}^1_{D/\Lambda}}(r,(d_i))$
in Theorem \ref{Thm: moduli of Lambda-tuples}
represents the \'etale sheafification of
$\overline{\mathcal M}^{\alpha',\beta,\gamma}_{\tilde{\Lambda}^1_{D/\Lambda}}
(r,(d_i))$.
For $(E_1,E_2,\Phi,F_*(E_1))\in
\overline{\mathcal M}^{\alpha',\beta,\gamma}
_{\tilde{\Lambda}^1_{D/\Lambda}}(r,(d_i))(T)$
and for $x,i$, take $p$ satisfying
$\alpha^{x,i}_0=\tilde{\alpha}_{p}$
and set
$l^x_i:=\ker(E_1|_{m_xx\times T}
\rightarrow E_1/F_p(E_1))$.
Let
\begin{equation*} \label {flag bundle functor}
 {\mathcal Flag}
 \longrightarrow 
 \overline{\mathcal M}^{\alpha',\beta,\gamma}
 _{\tilde{\Lambda}^1_{D/\Lambda}}(r,(d_i))
\end{equation*}
be the flag bundle functor
whose fiber over
$(E_1,E_2,\Phi,F_*(E_1))\in
\overline{\mathcal M}^{\alpha',\beta,\gamma}
_{\tilde{\Lambda}^1_{D/\Lambda}}(r,(d_i))(T)$
is the set of filtrations
\[
\Hom({\mathcal A}^x_i\otimes{\mathcal O}_T,l^x_i/l^x_{i+1})
\supset N^{x,i}_0\supset\cdots\supset N^{x,i}_{r^x_i m_x-1}
\qquad
(x\in D,\: 1\leq i\leq s_x)
\]
by ${\mathcal A}^x_i\otimes{\mathcal O}_T$-submodules.
Then a flag bundle
$\mathbf{Flag}\rightarrow
\overline{M}^{\alpha',\beta,\gamma}_{\tilde{\Lambda}^1_{D/\Lambda}}(r,(d_i))$
represents the \'etale sheafification of the morphism of functors
${\mathcal Flag}\rightarrow
\overline{\mathcal M}^{\alpha',\beta,\gamma}
_{\tilde{\Lambda}^1_{D/\Lambda}}(r,(d_i))$.
Note that $\mathbf{Flag}$ is a projective scheme over $\Lambda$,
since it is projective over
$\overline{M}^{\alpha',\beta,\gamma}
_{\tilde{\Lambda}^1_{D/\Lambda}}(r,(d_i))$
and
$\overline{M}^{\alpha',\beta,\gamma}_{\tilde{\Lambda}^1_{D/\Lambda}}(r,(d_i))$
is a projective scheme over $\Lambda$
by Theorem \ref{Thm: moduli of Lambda-tuples}.
The morphism \eqref{morphism to Lambda tuple functor}
is canonically extended to
a morphism of functors
\begin{equation}
\label {embedding to compactrification}
 {\mathcal M}^{\alpha}_{\mathrm{Hod}} \;\longrightarrow \;
 {\mathcal Flag}
\end{equation}
which is a locally closed immersion of functors.
So there is a locally closed subscheme
$M^{\alpha}_{\mathrm{Hod}}\subset\mathbf{Flag}$,
which represents the \'etale sheafification of \eqref {embedding to compactrification}.
So $M^{\alpha}_{\mathrm{Hod}}$ is the desired moduli space.
\hfill $\square$

\section{Deformation theory of $\lambda$-connections}
\label {sect: Defor lambda conn}

In this section, we will prove Theorem \ref{Thm: main theorem}.
Let $M^{\alpha}_{\mathrm{Hod}}(\lambda,\nu)$ be the fiber
of $M^{\alpha}_{\mathrm{Hod}}$ over a point $(\lambda,\nu)$ of $\Lambda$.
To describe infinitesimal deformations, 
it is technically more convenient to work 
with parabolic frames rather than parabolic framing structures; 
we thus adopt parabolic frames throughout this section.
Accordingly, we take an element 
$(E,\nabla,\{l^x_i,L^x_i\}) \in M^{\alpha}_{\mathrm{Hod}}(\lambda,\nu)$.

Now we consider the factorization of a linear endomorphism
associated with a $(\nu^{x,i}(w),{\mathcal A}^x_i)$-frame $ L^x_i$ for 
$(l^x_i/l^x_{i+1} , \overline{\nabla}^x_i)$.
Take a generator $\sigma^{x,i}$ of $L^x_i$.
So $\sigma^{x,i}\colon{\mathcal A}^x_i\xrightarrow{\sim} l^x_i/l^x_{i+1}$
is an isomorphism.
Consider the composition
\[
 w^x_i:=\sigma^{x,i}\circ ( w\cdot\mathrm{id}_{{\mathcal A}^x_i})
 \circ (\sigma^{x,i})^{-1}
 \colon l^x_i/l^x_{i+1} \longrightarrow l^x_i/l^x_{i+1}. 
\]
The transpose
\[
{}^tw^x_i\colon  (l^x_i/l^x_{i+1} )^{\vee}\longrightarrow (l^x_i/l^x_{i+1})^{\vee}
\]
induces an ${\mathcal O}_{m_xx}[T]$-module structure
on $(l^x_i/l^x_{i+1} )^{\vee}$ by
\[
{\mathcal O}_{m_xx}[T]\times (l^x_i/l^x_{i+1})^{\vee}
\,\ni\, (P(T),v) \; \mapsto \; P(\,{}^tw^x_i)v \,\in\, (l^x_i/l^x_{i+1})^{\vee},
\]
making $(l^x_i/l^x_{i+1})^{\vee}$ isomorphic to ${\mathcal A}^x_i$ 
as an ${\mathcal O}_{m_xx}[T]$-module.
So there is an isomorphism
\[
 \check{\sigma}^{x,i}\colon {\mathcal A}^x_i \: \xrightarrow{\;\sim\;} \:
 (l^x_i/l^x_{i+1})^{\vee}
\]
of ${\mathcal O}_{m_xx}[T]$-modules.
Consider the composition
\[
 \theta^x_i:= \sigma^{x,i}\circ(\check{\sigma}^{x,i})^{-1}
 \colon \: (l^x_i/l^x_{i+1})^{\vee}
 \xrightarrow[\sim]{(\check{\sigma}^{x,i})^{-1}} {\mathcal A}^x_i
 \xrightarrow[\sim]{\;\sigma^{x,i}\:} l^x_i/l^x_{i+1}.
\]
Then $w^x_i$ factors through a homomorphism
$\kappa^x_i\colon l^x_i/l^x_{i+1}\rightarrow (l^x_i/l^x_{i+1})^{\vee}$:
\[
 w^x_i
 =
 \theta^x_i\circ\kappa^x_i\colon\:
  l^x_i/l^x_{i+1} \xrightarrow{\;\kappa^x_i\:}
  (l^x_i/l^x_{i+1})^{\vee}
 \xrightarrow[\sim]{\;\theta^x_i\:}
 l^x_i/l^x_{i+1}.
\]
We can see $^t\theta^x_i=\theta^x_i$
and $^t\kappa^x_i=\kappa^x_i$ 
by the same argument as in \cite[Proposition 1.1]{Inaba2}.

A similar argument as in \cite[Proposition 2.14]{Inaba2} yields the following proposition.

\begin{Prop}
There is a bijection between the set of
$(\nu^{x,i}(w),{\mathcal A}^x_i)$-frames
$L^x_i$
on $(l^x_i/l^x_{i+1},\overline{\nabla}^x_i)$
and the set of pairs $(\theta^x_i,\kappa^x_i)$
of an isomorphism 
$\theta^x_i\colon (l^x_i/l^x_{i+1})^{\vee}\xrightarrow{\sim} l^x_i/l^x_{i+1}$
and a homomorphism
$\kappa^x_i\colon  l^x_i/l^x_{i+1}\rightarrow ( l^x_i/l^x_{i+1})^{\vee}$
which satisfy
$^t\theta^x_i=\theta^x_i$,
$^t\kappa^x_i=\kappa^x_i$,
$(\theta^x_i\circ\kappa^x_i)^{r^x_i}=z\cdot\mathrm{id}_{l^x_i/l^x_{i+1}}$
and $\overline{\nabla}^x_i=\nu^{x,i}(\theta^x_i\circ\kappa^x_i)$.
\end{Prop}

Set
\begin{align*}
\mathrm{Sym}^2\left(l^x_i/l^x_{i+1}\right)
&:=
\left\{
\zeta\in \Hom \big( (l^x_i/l^x_{i+1})^{\vee},l^x_i/l^x_{i+1}\big)
\,\middle|\,
^t\zeta=\zeta
\right\},
\\
\mathrm{Sym}^2\left((l^x_i/l^x_{i+1})^{\vee}\right)
&:=
\left\{
\xi\in \Hom \big( l^x_i/l^x_{i+1},(l^x_i/l^x_{i+1})^{\vee}\big)
\,\middle|\,
^t\xi=\xi
\right\}.
\end{align*}
Define sheaves ${\mathcal G}^0$, ${\mathcal G}^1$,
$\overline{S\,}^2(E|_D)$, $\overline{S\,}^2(E|_D^{\vee})$,
$A^0$, $A^2$ and $\overline{\mathcal G\,}^1$  on $X$ by
\begin{align*}
&{\mathcal G}^0
:=
\left\{ u\in{\mathcal End}(E)\,\middle|\,
\begin{array}{l}
\text{$u|_{m_xx}(l^x_i)\subset l^x_i$ 
for $x\in D$ and $1\leq i\leq s_x$}
\end{array}
\right\},
\\
&{\mathcal G}^1
:=
\left\{ v\in {\mathcal End}(E)\otimes\Omega^1_X(D)
\,\middle|\,
\begin{array}{l}
\text{$v|_{m_xx}(l^x_i)\subset l^x_i\otimes\Omega^1_X(D)|_{m_xx}$}
\\
\text{for $x\in D$ and $1\leq i\leq s_x$}
\end{array}
\right\},
\\
&\overline{S\,}^2(E|_D)
:=
\bigoplus_{x\in D} \bigoplus_{i=1}^{s_x}\mathrm{Sym}^2\left(l^x_i/l^x_{i+1}\right),
\quad
\overline{S\,}^2(E|_D^{\vee})
:=
\bigoplus_{x\in D} \bigoplus_{i=1}^{s_x}
\mathrm{Sym}^2\left((l^x_i/l^x_{i+1})^{\vee}\right),
\\
&A^0:=\bigoplus_{x\in D}\bigoplus_{i=1}^{s_x} 
{\mathcal A}^x_i,
\qquad
A^2:=
\Hom(A^0,\Omega^1_X(D)|_D), \ \ \text{and}
\\
&\overline{\mathcal G\,}^1:=
 \bigoplus_{x\in D}\bigoplus_{i=1}^{s_x}
 \End(l^x_i/l^x_{i+1})\otimes\Omega^1_X(D)|_{m_xx},
\end{align*}
respectively.  
Then there is a surjection
\[
 {\mathcal G}^1\ni  v  \longmapsto 
 \big(\,(v|_{m_xx})^x_i\,\big) \in \overline{\mathcal G\,}^1,
\]
where 
$(v|_{m_xx})^x_i\in \End(l^x_i/l^x_{i+1})\otimes\Omega^1_X(D)|_{m_xx}$
is the homomorphism induced by
$v|_{m_xx}\in \End(E|_{m_xx})\otimes\Omega^1_X(D)|_{m_xx}$.
Define homomorphisms
\[
\tau_{1}\colon 
{\mathcal G}^0\oplus A^0 \longrightarrow \overline{S\,}^2(E|_D),
\qquad
\tau_{2}\colon 
{\mathcal G}^0\oplus A^0 \longrightarrow \overline{S\,}^2(E|_D^{\vee})
\]
by setting
\begin{align*}
\tau_{1}(u,(P^x_i(w))) 
&:=
\big( \, -(u|_{m_xx})^x_i\circ\theta^x_i
-\theta^x_i\circ\!\,^t(u|_{m_xx})^x_i+\theta^x_i\circ P^x_i(\!\,^tw^x_i) \, \big),
\\
\tau_{2}(u,(P^x_i(w))) 
&:=
\big( \, \kappa^x_i\circ (u|_{m_xx})^x_i+\!\,^t(u|_{m_xx})^x_i\circ\kappa^x_i
-P^x_i(\!\,^tw^x_i)\circ\kappa^x_i \, \big)
\end{align*}
for $(u,(P^x_i(w)))\in {\mathcal G}^0\oplus A^0$.
Define a homomorphism
\[
\eta\colon \overline{S\,}^2(E|_D)\oplus \overline{S\,}^2(E|_D^{\vee})
\longrightarrow \overline{\mathcal G\,}^1
\]
by setting
\begin{align*}
\eta((\zeta^x_i),(\xi^x_i))
&:=
\bigg(\,
\sum_{k=1}^{r^x_i-1}\sum_{l=1}^k \nu^{x,i}_k(z)\: 
(w^x_i)^{k-l}\circ(\theta^x_i\circ\xi^x_i+\zeta^x_i\circ\kappa^x_i)\circ(w^x_i)^{l-1}
\bigg)
\end{align*}
for $((\zeta^x_i),(\xi^x_i))
\in \overline{S\,}^2(E|_D)\oplus \overline{S\,}^2(E|_D^{\vee})$.
Let
\[
\Theta\colon \overline{S\,}^2(E|_D)\oplus \overline{S\,}^2(E|_D^{\vee})
\longrightarrow A^2
\]
be the homomorphism defined by
\[
\Theta((\zeta^x_i),(\xi^x_i)) \colon
A^0 \, \ni \, (P^x_i(w)) \; \mapsto \;
\big( \Tr(P^x_i(w^x_i)\circ (\theta^x_i\circ\xi^x_i+\zeta^x_i\circ\kappa^x_i))\big)
\frac{\mathrm{d}z}{z^{m_x}}
\,\in\, \Omega^1_X(D)|_D
\]
for $((\zeta^x_i),(\xi^x_i))\in 
\overline{S\,}^2(E|_D)\oplus \overline{S\,}^2(E|_D^{\vee})$.

Define a complex 
\begin{equation} \label {complex defining tangent space}
{\mathcal F}^{\bullet}\colon \;
{\mathcal F}^0 \xrightarrow{\;d^0_{\mathcal F^{\bullet}}\:}
{\mathcal F}^1 \xrightarrow{\;d^1_{\mathcal F^{\bullet}}\:}
{\mathcal F}^2
\end{equation}
of sheaves on $X$ by setting
\[
{\mathcal F}^0:=
{\mathcal G}^0\oplus A^0,
\qquad
{\mathcal F}^1:=
{\mathcal G}^1\oplus \overline{S\,}^2(E|_D)\oplus \overline{S\,}^2(E|_D^{\vee}),
\qquad
{\mathcal F}^2:=
\overline{\mathcal G\,}^1 \oplus A^2,
\]
\[
d^0_{\mathcal F^{\bullet}}(u,(P^x_i(w)))
=
\big(\nabla u-u \nabla , \, \tau_1(u,(P^x_i(w))),
\, \tau_2(u,(P^x_i(w)))\big)
\]
for $(u,(P^x_i(w))) \in {\mathcal G}^0\oplus A^0$ and
\[
d^1_{\mathcal F^{\bullet}}(v,(\zeta^x_i),(\xi^x_i))
=
\big( \big(\,(v|_{m_xx})^x_i\,\big)-\eta \big((\zeta^x_i),(\xi^x_i)\big),\,
\Theta((\zeta^x_i),(\xi^x_i)) \big)
\]
for $(v,(\zeta^x_i),(\xi^x_i))\in
{\mathcal G}^1\oplus \overline{S\,}^2(E|_D)\oplus\overline{S\,}^2(E|_D^{\vee})$.

\subsection{Relative tangent spaces of $M^{\alpha}_{\mathrm{Hod}}$
over $\Lambda$} \label{subsect relative tangent sp}

\begin{Prop} \label {Prop: tangent space}
Let $M^{\alpha}_{\mathrm{Hod}}(\lambda,\nu)$
be the fiber of
$M^{\alpha}_{\mathrm{Hod}}
\rightarrow \Lambda$
over $(\lambda,\nu)\in\Lambda$.
Then the  tangent space of $M^{\alpha}_{\mathrm{Hod}}(\lambda,\nu)$
at $\big(E,\nabla,\{l^x_i,L^x_i\}\big)$
is isomorphic to
$\mathbb{H}^1({\mathcal F}^{\bullet})$.
\end{Prop}

\begin{proof}
Set $L^x_i= {\mathcal A}^x_i \sigma^{x,i}$.
Assume that  an element
$v\in\mathbb{H}^1({\mathcal F}^{\bullet})$
is given.
Take an affine open covering
$\{U_{\alpha}\}$ of $X$
such that
$E|_{U_{\alpha}}\cong{\mathcal O}_{U_{\alpha}}^{\oplus r}$.
We may assume that
each point $x \in D$ is contained in a unique open set $U_\alpha$.
Then $v$ is expressed as
a \v{C}ech cohomology class
\[
 v=\big[
 \{u_{\alpha\beta}\},\{v_{\alpha}\},
 \{(\zeta^{x}_{i,\alpha}),(\xi^{x}_{i,\alpha})\}
 \big],
\]
where
$
\{u_{\alpha\beta}\}\in C^1\big(\{U_{\alpha}\},{\mathcal G}^0\big)
$,
$\{v_{\alpha}\}\in C^0\big(\{U_{\alpha}\},{\mathcal G}^1\big)$,
$(\zeta^{x}_{i,\alpha})\in C^0\big(\{U_{\alpha}\},\overline{S\,}^2(E|_D)\big)$
and
$(\xi^{x}_{i,\alpha})\in C^0\big(\{U_{\alpha}\},\overline{S\,}^2(E|_D^{\vee})\big)$
satisfy
\begin{align*}
&u_{\beta\gamma}-u_{\alpha\gamma}+u_{\alpha\beta}=0,
\qquad
v_{\beta}-v_{\alpha}=\nabla u_{\alpha\beta}-u_{\alpha\beta}\nabla,
\\
&((v_{\alpha}|_{m_xx})^x_i)
=\eta((\zeta^x_{i,\alpha}),(\xi^x_{i,\alpha}))|_{U_{\alpha}},
\quad
\Theta( (\zeta^x_{i,\alpha}), (\xi^x_{i,\alpha}) )=0.
\end{align*}
For a scheme $Y$ over $\mathbb{C}$,
write $Y[\overline{t}]:=Y\times\mathbb{C}[t]/(t^2)$.
Take a free ${\mathcal O}_{U_{\alpha}[\overline{t}]}$-module
$E_{\alpha}$ with an isomorphism
$E_{\alpha}\xrightarrow[\sim]{\phi_{\alpha}} 
E|_{U_{\alpha}}\otimes\mathbb{C}[t]/(t^2)$
and filtrations
$(l_{\alpha})^x_i$ of $E_{\alpha}|_{m_xx[\overline{t}]}$
corresponding to $l^x_i|_{U_{\alpha}}\otimes\mathbb{C}[t]/(t^2)$
through $\phi_{\alpha}$.
Put
\begin{align*}
 &\theta_{\alpha\beta}:=\phi_{\alpha}^{-1} \circ
 (\mathrm{id}+\overline{t}u_{\alpha\beta})\circ\phi_{\beta}
 \colon
 E_{\beta}|_{U_{\alpha\beta}[\overline{t}]}
 \xrightarrow{\;\sim\;}
 E_{\alpha}|_{U_{\alpha\beta}[\overline{t}]} \ \ \text{and}
\\
&\nabla_{\alpha}:=
\phi_{\alpha}^{-1}\circ
 (\nabla+\overline{t}v_{\alpha})\circ\phi_{\alpha}
 \colon
 E_{\alpha} \longrightarrow
 E_{\alpha}\otimes\Omega^1_{U_{\alpha}[\overline{t}]/\Spec\mathbb{C}[t]/(t^2)}
 ((D\cap U_{\alpha})[\overline{t}]).
\end{align*}
Then we have
\begin{align*}
 \theta_{\gamma\alpha}^{-1}\circ\theta_{\gamma\beta}\circ\theta_{\beta\alpha}
 &=
 \phi_{\alpha}^{-1}(\mathrm{id}-\overline{t}u_{\gamma\alpha})\phi_{\gamma} 
 \circ \phi_{\gamma}^{-1}(\mathrm{id}+\overline{t}u_{\gamma\beta})\phi_{\beta}
 \circ
 \phi_{\beta}^{-1}(\mathrm{id}+\overline{t}u_{\beta\alpha})\phi_{\alpha}
 =\mathrm{id}\ \ \text{and}
\\
\theta_{\alpha\beta}\nabla_{\beta}-\nabla_{\alpha}\theta_{\alpha\beta}
&=
\phi_{\alpha}^{-1}(\mathrm{id}+\overline{t}u_{\alpha\beta})\phi_{\beta}
\phi_{\beta}^{-1}(\nabla+\overline{t}v_{\beta})\phi_{\beta}
-\phi_{\alpha}^{-1}(\nabla+\overline{t}v_{\alpha})\phi_{\alpha}
\phi_{\alpha}^{-1}(\mathrm{id}+\overline{t}u_{\alpha\beta})\phi_{\beta}
\\
&=
\phi_{\alpha}^{-1}(u_{\alpha\beta}\nabla+v_{\beta})\overline{t}\phi_{\beta}
-\phi_{\alpha}^{-1}(\nabla u_{\alpha\beta}+v_{\alpha})\overline{t}\phi_{\beta}
=0.
\end{align*}
So the patching data $\{E_{\alpha},\theta_{\alpha\beta},\nabla_{\alpha}\}$
gives rise to a tuple
$
(E_v,\nabla_v,\{(l_v)^x_i\})
$
of a vector bundle $E_v$ on $X[\overline{t}]$, a connection
$\nabla_v\colon E_v\rightarrow E_v\otimes
\Omega^1_{X[\overline{t}]/\Spec\mathbb{C}[t]/(t^2)}(D[\overline{t}])$
and
filtrations $(l_v)^x_i$ on $E_v|_{m_xx[\overline{t}]}$ 
preserved by $\nabla|_{m_xx[\overline{t}]}$. 
Consider the lifts
\begin{align*}
&(\theta_v)^x_i:=\theta^x_i+\overline{t}\zeta^x_i\colon
\big((l_v)^x_i/(l_v)^x_{i+1}\big)^{\vee} \longrightarrow (l_v)^x_i/(l_v)^x_{i+1}
\;\text{and}
\\
&(\kappa_v)^x_i:=\kappa^x_i+\overline{t}\xi^x_i\colon
(l_v)^x_i/(l_v)^x_{i+1} \longrightarrow \big((l_v)^x_i/(l_v)^x_{i+1}\big)^{\vee}
\end{align*}
of 
$\theta^x_i$ and $\kappa^x_i$,
respectively,
where we set
$\xi^x_i:=\xi^x_{i,\alpha}$, $\zeta^x_i:=\zeta^x_{i,\alpha}$
for $U_{\alpha}$ containing $x$.
Then
\[
(w_v)^x_i:=
(\theta_v)^x_i\circ(\kappa_v)^x_i
=
\theta^x_i\circ\kappa^x_i+\overline{t}(\theta^x_i\circ\xi^x_i+\zeta^x_i\circ\kappa^x_i)
\colon (l_v)^x_i/(l_v)^x_{i+1} \longrightarrow (l_v)^x_i/(l_v)^x_{i+1}
\]
is a lift of $w^x_i=\theta^x_i\circ\kappa^x_i$.
Since $\Theta( (\zeta^x_i), (\xi^x_i) )=0$,
we have $\Tr((w^x_i)^k(\theta^x_i\circ\xi^x_i+\zeta^x_i\circ\kappa^x_i))=0$
for $0\leq k\leq r^x_i-1$.
By \cite[Lemma 1.4]{Inaba2}, there is an endomorphism $\psi$ of 
$l^x_i/l^x_{i+1}$
such that
$\theta^x_i\circ\xi^x_i+\zeta^x_i\circ\kappa^x_i
=\psi\circ w^x_i-w^x_i\circ\psi$.
Since $(w^x_i)^{r^x_i}=z$, we have
\begin{align*}
\big((w_v)^x_i\big)^{r^x_i}
&=
(w^x_i)^{r^x_i}+\overline{t}\sum_{l=0}^{r^x_i-1} (w^x_i)^l
(\psi\circ w^x_i-w^x_i\circ\psi)(w^x_i)^{r^x_i-l-1}
\\
&=z
+\overline{t}(\psi\circ(w^x_i)^{r^x_i}-(w^x_i)^{r^x_i}\circ\psi)
=z.
\end{align*}
Since
$((v_{\alpha}|_{m_xx})^x_i)=\eta((\zeta^x_i),(\xi^x_i))|_{U_{\alpha}}$,
we have
\begin{align} \label {local deformation of connection}
\begin{split}
\sum_{k=0}^{r^x_i-1}
\nu^{x,i}_k(z)((w_v)^x_i)^k
&=
\sum_{k=0}^{r^x_i-1}\nu^{x,i}_k(z)\big(w^x_i+\overline{t}
(\theta^x_i\circ\xi^x_i+\kappa^x_i\circ\zeta^x_i)\big)^k
\\
&=
\sum_{k=0}^{r^x_i-1}\nu^{x,i}_k(z)(w^x_i)^k
\\
&\qquad
+
\overline{t}\sum_{k=1}^{r^x_i-1}\sum_{l=1}^k
\nu^{x,i}_k(z)(w^x_i)^{k-l}\circ(\theta^x_i\circ\xi^x_i+\zeta^x_i\circ\kappa^x_i)
\circ(w^x_i)^{l-1}
\\
&=\big(\overline{\nabla}_v\big)^x_i
\end{split}
\end{align}
for $U_{\alpha}$ containing $x$.
Set 
${\mathcal A}^x_i[\overline{t}]:=
{\mathcal A}^x_i\otimes\mathbb{C}[t]/(t^2)$
and take a lift 
$\tilde{u}\in (l_v)^x_i/(l_v)^x_{i+1}$ of $\sigma^{x,i}(1)\in l^x_i/l^x_{i+1}$.
Then the homomorphism
$
\sigma^{x,i}[v]\colon {\mathcal A}^x_i[\overline{t}]
\rightarrow (l_v)^x_i/(l_v)^x_{i+1}
$
defined by
\[
 \sigma^{x,i}[v]\Big(
 \sum_{0\leq k\leq r^x_i-1}b_k(z)w^k\Big)=
 \sum_{0\leq k\leq r^x_i-1} b_k(z)\big((w_v)^x_i\big)^k\tilde{u}
 \qquad
 (b_k(z)\in{\mathcal O}_{m_xx[\overline{t}]})
\]
satisfies
$\sigma^{x,i}[v]\circ w\cdot\mathrm{id}_{{\mathcal A}^x_i[\overline{t}]}
=(w_v)^x_i\circ \sigma^{x,i}[v]$.
Combining with the equality \eqref{local deformation of connection},
we have
$(\sigma^{x,i}[v]\otimes\mathrm{id})\circ
\nu^{x,i}(w)\mathrm{id}_{{\mathcal A}^x_i[\overline{t}]}
=(\overline{\nabla}_v)^x_i\circ \sigma^{x,i}[v]$.
Thus 
$(E_v,\nabla_v,\{(l_v)^x_i, 
{ \mathcal A}^x_i[\overline{t}] \sigma^{x,i}[v]  \})$
becomes a lift of
$(E,\nabla,\{l^x_i, {\mathcal A}^x_i \sigma^{x,i} \})$
and determines an element of 
${\mathcal M}^{\alpha}_{\mathrm{Hod}}(\Spec\mathbb{C}[t]/(t^2))$.
So we get a tangent vector of $M^{\alpha}_{\mathrm{Hod}}(\lambda,\nu)$.
The converse direction is proved in the same way as in 
\cite[Proposition 2.17]{Inaba2}.
\end{proof}

In order to calculate $\mathbb{H}^1({\mathcal F}^{\bullet})$
for the complex ${\mathcal F}^{\bullet}$ in \eqref{complex defining tangent space},
consider the complexes
\[
{\mathcal F}_0^{\bullet}:=
[{\mathcal G}^0\oplus A^0 \rightarrow \overline{S\,}^2(E|_D)],
\qquad
{\mathcal F}_1^{\bullet}:=
[{\mathcal G}^1\oplus \overline{S\,}^2(E|_D^{\vee})
\rightarrow \overline{\mathcal G\,}^1\oplus A^2].
\]
Then there is an exact sequence of complexes
\begin{equation*} \label {decomposition of the complex}
0\longrightarrow {\mathcal F}_1^{\bullet}[-1] \longrightarrow
{\mathcal F}^{\bullet} \longrightarrow {\mathcal F}_0^{\bullet}
\longrightarrow 0,
\end{equation*}
which induces the long exact sequence
\begin{align} \label {long exact sequence for tangent}
\begin{split}
&0 \longrightarrow \mathbb{H}^0({\mathcal F}^{\bullet})
\longrightarrow \mathbb{H}^0({\mathcal F}_0^{\bullet})
\longrightarrow \mathbb{H}^0({\mathcal F}_1^{\bullet})
\longrightarrow \mathbb{H}^1({\mathcal F}^{\bullet})
\\
& \hspace{100pt}
\longrightarrow \mathbb{H}^1({\mathcal F}_0^{\bullet})
\longrightarrow \mathbb{H}^1({\mathcal F}_1^{\bullet})
\longrightarrow \mathbb{H}^2({\mathcal F}^{\bullet})
\longrightarrow 0.
\end{split}
\end{align}
An element $v\in\mathbb{H}^1({\mathcal F}^{\bullet})$
is expressed as
\[
v=[\{u_{\alpha\beta}\},\{v_{\alpha},\rho_{\alpha}\}],
\]
where 
$\{U_{\alpha}\}$ is an affine open covering of $X$,
$\{u_{\alpha\beta}\}\in C^1(\{U_{\alpha}\},{\mathcal G}^0)$,
$\{v_{\alpha}\}\in C^0(\{U_{\alpha}\},{\mathcal G}^1)$,
$\{\rho_{\alpha}\}\in 
C^0\big(\{U_{\alpha}\},\overline{S\,}^2(E|_D)
\oplus\overline{S\,}^2(E|_D^{\vee})\big)$
and $\rho_{\alpha}=(\zeta^x_{i,\alpha},\xi^x_{i,\alpha})$
such that
\begin{align*}
&\nabla u_{\alpha\beta}-u_{\alpha\beta}\nabla=v_{\beta}-v_{\alpha},
\quad
(v_{\alpha})^x_i=
\sum_{k=1}^{r^x_i-1}\sum_{l=1}^k \nu^{x,i}_k(z)
(w^x_i)^{k-l}(\theta^x_{i}\xi^x_{i,\alpha}
+\zeta^x_{i,\alpha}\kappa^x_{i})(w^x_i)^{l-1},
\\
& \Tr\big( (w^x_i)^k\circ(\theta^x_i\circ\xi^x_{i,\alpha}+\zeta^x_{i,\alpha}\circ\kappa^x_i)\big)=0
\qquad (0\leq k\leq r^x_i-1).
\end{align*}
For 
$\rho=((\zeta^x_i),(\xi^x_i)),
\rho'=((\zeta'^x_i),(\xi'^x_i))
\in \overline{S\,}^2(E|_D)\oplus\overline{S\,}^2(E|_D^{\vee})$,
put
\[
\Xi(\rho,\rho'):=
\frac{1}{2} \Big(  \sum_{i=1}^{s_x}\sum_{k=1}^{r^x_i-1}
\sum_{l=1}^k
\nu^{x,i}_k(z)
\Tr\big(
\zeta'^x_i\,^t(w^x_i)^{k-l}\xi^x_i(w^x_i)^{l-1}
-(w^x_i)^{k-l}\zeta^x_i(\,^tw^x_i)^{l-1}\xi'^x_i
\big)\Big),
\]
which is in $\Omega^1_X(D)|_D$.
Define a bilinear pairing
\begin{equation}
\label {nondegenerate pairing}
 \omega_{\mathbb{H}^1({\mathcal F}^{\bullet})}
 \colon \mathbb{H}^1({\mathcal F}^{\bullet}) \times\mathbb{H}^1({\mathcal F}^{\bullet})
 \:\longrightarrow\:
 \mathbb{H}^2\big({\mathcal O}_X\!\xrightarrow{d}\!\Omega^1_X(D)\!
 \rightarrow\!\Omega^1_X(D)|_D\big)
 \!=\!H^2(X,\mathbb{C})=\mathbb{C}
\end{equation}
by setting
\begin{align*}
&\omega_{\mathbb{H}^1({\mathcal F}^{\bullet})}
\big([\{u_{\alpha\beta}\},\{v_{\alpha},\rho_{\alpha}\}],
[\{u'_{\alpha\beta}\},\{v'_{\alpha},\rho'_{\alpha}\}]\big)
\\
&\qquad
=
\big[
\{\Tr(u_{\alpha\beta}u'_{\beta\gamma})\},
\{-\Tr(u_{\alpha\beta}v'_{\beta}-v_{\alpha}u'_{\alpha\beta})\},
\{\Xi(\rho_{\alpha},\rho'_{\alpha})\}
\big]
\end{align*}
for 
$[\{u_{\alpha\beta}\},\{v_{\alpha},\rho_{\alpha}\}],
[\{u'_{\alpha\beta}\},\{v'_{\alpha},\rho'_{\alpha}\}]
\in\mathbb{H}^1({\mathcal F}^{\bullet})$.
We can see by the same argument as in 
\cite[Section 2.5]{Inaba2} that
the pairing \eqref{nondegenerate pairing}
is well-defined.

\begin{Lem} \label {lemma: nondegenerate}
The pairing $\omega_{\mathbb{H}^1({\mathcal F}^{\bullet})}$
in \eqref {nondegenerate pairing} is nondegenerate.
\end{Lem}

\begin{proof}
The bilinear pairing \eqref{nondegenerate pairing}
induces a homomorphism
$b\colon \mathbb{H}^1({\mathcal F}^{\bullet})
\rightarrow \mathbb{H}^1({\mathcal F}^{\bullet})^{\vee}$.
We have to prove that $b$ is an isomorphism.
There is an exact commutative diagram
\begin{equation} \label {diagram for five lemma}
\begin{CD}
\mathbb{H}^0({\mathcal F}_0^{\bullet}) @>>> \mathbb{H}^0({\mathcal F}_1^{\bullet})
@>>> \mathbb{H}^1({\mathcal F}^{\bullet}) @>>>
\mathbb{H}^1({\mathcal F}_0^{\bullet}) @>>>
\mathbb{H}^1({\mathcal F}_1^{\bullet}) 
\\
@V b_1 VV @V b_2 VV @V b VV @V ^tb_2 VV @V ^tb_1 VV
\\
\mathbb{H}^1({\mathcal F}_1^{\bullet})^{\vee} @>>> 
\mathbb{H}^1({\mathcal F}_0^{\bullet})^{\vee}
@>>> \mathbb{H}^1({\mathcal F}^{\bullet})^{\vee} @>>>
\mathbb{H}^0({\mathcal F}_1^{\bullet})^{\vee} @>>>
\mathbb{H}^0({\mathcal F}_0^{\bullet})^{\vee}\rlap{,} 
\end{CD}
\end{equation}
where $b_2$, $^tb_2$ are given by the pairing
\begin{equation} \label {pairing for tau_2}
 \mathbb{H}^0({\mathcal F}_1^{\bullet})\times \mathbb{H}^1({\mathcal F}_0^{\bullet})
 \:\longrightarrow\: 
 \mathbb{H}^2\big({\mathcal O}_X\!\xrightarrow{d}\!\Omega^1_X(D)\!
 \rightarrow\!\Omega^1_X(D)|_D\big)
 \!=\!H^2(X,\mathbb{C})=\mathbb{C}
\end{equation}
defined by
\[
([\{v_{\alpha}\},\{(\xi^x_{i,\alpha})\}],\,
[\{u'_{\alpha\beta}\},\{(\zeta'^x_{i,\alpha})\}])
\mapsto
[\{\Tr(v_{\alpha}u'_{\alpha\beta})\},
\{\Xi((0,\xi^x_{i,\alpha}),(\zeta'^x_{i,\alpha},0))\}],
\]
and $b_1$, $^tb_1$ are given by the pairing
\begin{equation} \label {pairing for tau_1}
 \mathbb{H}^0({\mathcal F}_0^{\bullet})\times \mathbb{H}^1({\mathcal F}_1^{\bullet})
 \:\longrightarrow\: 
 \mathbb{H}^2\big({\mathcal O}_X\!\xrightarrow{d}\!\Omega^1_X(D)\!
 \rightarrow\!\Omega^1_X(D)|_D\big)
 \!=\!H^2(X,\mathbb{C})=\mathbb{C}
\end{equation}
defined by
\begin{align*}
&\big([\{u_{\alpha}\},\{(P^x_{i,\alpha}(w))\}],\,
[\{v'_{\alpha\beta}\},\{(v'^x_{i,\alpha}),b^x_{i,\alpha})\}]\big)
\\
&\qquad
\mapsto \;
-\big[ \{\Tr(u_{\alpha}v'_{\alpha\beta})\},
\big\{\Tr(u_{\alpha}v'^x_{i,\alpha})+
\frac{1}{2}b^x_{i,\alpha}\big(a^{x,i}(w)'P^x_{i,\alpha}(w)\big)\big\}\big]
\end{align*}
with $a^{x,i}(w)':=\sum_{k=1}^{r^x_i-1}ka^{x,i}_k(z)w^{k-1}$.
Here $a^{x,i}_k(z)$ are the polynomials in
Definition \ref{def: generic ramified local exponent data intro}.
The pairing \eqref{pairing for tau_2} is nondegenerate,
since $b_2$ fits into the exact commutative diagram
\[
\xymatrix@=1.5pc{
0\rightarrow \mathbb{H}^0({\mathcal G}^1\rightarrow\overline{\mathcal G^1}) 
\ar[r] \ar[d]^{\cong} & 
\mathbb{H}^0({\mathcal F}_1^{\bullet}) \ar[r] \ar[d]^{b_2} 
& \ker(\overline{S\,}^2(E|_D^{\vee})\rightarrow A^2) \ar[r] \ar[d]^{\cong}
& \mathbb{H}^1({\mathcal G}^1\rightarrow\overline{{\mathcal G}^1}) \ar[d]^{\cong}
\\
0\longrightarrow H^1({\mathcal G}^0)^{\vee} \ar[r] &
\mathbb{H}^1({\mathcal F}_0^{\bullet})^{\vee}  \ar[r] &
\coker(A^0\rightarrow\overline{S\,}^2(E|_D))^{\vee}  \ar[r]
& H^0({\mathcal G}^0)^{\vee}\rlap{.}
}
\]
The pairing \eqref{pairing for tau_1} is also nondegenerate,
since $b_1$ fits into the exact commutative diagram
\[
\xymatrix@=1.5pc{
\ker(A^0\rightarrow \overline{S\,}^2(E|_D))=0 \ar[r] \ar[d]^{\cong} &
\mathbb{H}^0({\mathcal F}_0^{\bullet}) \ar[r] \ar[d]^{b_1} &
H^0({\mathcal G}^0) \ar[r] \ar[d]^{\cong} &
\coker(A^0\rightarrow \overline{S\,}^2(E|_D)) \ar[d]^{\cong}
\\
\coker(\overline{S\,}^2(E|_D^{\vee})\rightarrow A^2)^{\vee}=0 \ar[r] &
\mathbb{H}^1({\mathcal F}_1^{\bullet})^{\vee} \ar[r] &
\mathbb{H}^1({\mathcal G}^1\rightarrow\overline{\mathcal G\,}^1)^{\vee} \ar[r] &
\ker(\overline{S\,}^2(E|_D^{\vee})\rightarrow A^2)^{\vee}\rlap{.}
}
\]
The diagram \eqref {diagram for five lemma} and the five lemma yield that 
$b\colon \mathbb{H}^1({\mathcal F}^{\bullet})\rightarrow
\mathbb{H}^1({\mathcal F}^{\bullet})^{\vee}$
is an isomorphism.
\end{proof}

\begin{Prop} \label {Prop: smoothness of moduli}
The moduli space
$M^{\alpha}_{\mathrm{Hod}}$
is smooth over $\Lambda$.
\end{Prop}

\begin{proof}
Set
$
\Lambda_1:=
\big\{
(\lambda,(\eta_x)_{x\in D})
\,\big|\,
\lambda\in\mathbb{C},\;
\eta_x\in\Omega^1_X(D)|_{m_xx}, \;
\lambda d+\sum_{x\in D}\mathrm{res}_x(\eta_x)=0
\big\}
$
and consider the map
\[
\Lambda \, \ni \, (\lambda,(\nu^{x,i}(w))) 
\; \longmapsto \; 
\Big( \lambda, \Big(\sum_{i=1}^{s_x} r^x_i\nu^{x,i}_0(z)\Big)\Big)
\, \in \, \Lambda_1,
\]
which is smooth and surjective.
Let
$N\rightarrow \Lambda_1$
be the relative moduli space of
pairs $(L,\nabla_L)$
of a line bundle $L$ on $X$ and a $\lambda$-connection
$\nabla_L\colon L\rightarrow L\otimes\Omega^1_X(D)$
satisfying $\nabla_L|_{m_xx}=\eta_x\mathrm{id}_{L|_{m_xx}}$,
where  $(\lambda,(\eta_x))\in \Lambda_1$.
Then $N$ is smooth over $\Lambda_1$, since it is an affine space bundle
over a family of Jacobian varieties over $\Lambda_1$.
Let
\[
\det\colon M^{\alpha}_{\mathrm{Hod}} 
\longrightarrow N\times_{\Lambda_1}\Lambda
\]
be the morphism defined by $\det(E,\nabla,\{l^x_i,L^x_i\})=\det (E,\nabla)$.
In order to prove the smoothness of
$M^{\alpha}_{\mathrm{Hod}}$ over $\Lambda$,
it is sufficient to prove that
$\det$ is a smooth morphism.
For $(E,\nabla,\{l^x_i, {\mathcal A}^x_i \sigma^{x,i}   \})\in M^{\alpha}_{\mathrm{Hod}}$,
set
\begin{align*}
&{\mathcal G}_{\mathfrak{sl}}^0:=
\big\{ u\in{\mathcal G}^0\,\big|\, \Tr(u)=0 \big\},
\quad
{\mathcal G}_{\mathfrak{sl}}^1:=
\big\{ v\in{\mathcal G}^1\,\big|\, \Tr(v)=0 \big\},
\\
&\overline{\mathcal G}_{\mathfrak{sl}}^1:=
\big\{ (v^x_i)\in \overline{\mathcal G\,}^1\,\big|\,
\sum_{i=1}^{s_x}\Tr(v^x_i)=0 \; (\forall x\in D) \big\},
\end{align*}
and define the complex ${\mathcal F}_{\mathfrak{sl}}^{\bullet}$ by
\[
{\mathcal F}_{\mathfrak{sl}}^0
:=
{\mathcal G}_{\mathfrak{sl}}^0\oplus A^0,
\quad
{\mathcal F}_{\mathfrak{sl}}^1
:=
{\mathcal G}_{\mathfrak{sl}}^1\oplus
\overline{S\,}^2(E|_D)\oplus\overline{S\,}^2(E|_D^{\vee}),
\quad
{\mathcal F}_{\mathfrak{sl}}^2
:=
\overline{\mathcal G}^1_{\mathfrak{sl}}\oplus A^2,
\]
and
\[
d^0_{{\mathcal F}^{\bullet}_{\mathfrak{sl}}}
:=d^0_{{\mathcal F}^{\bullet}}|_{{\mathcal F}^0_{\mathfrak{sl}}}\colon
{\mathcal F}_{\mathfrak{sl}}^0\longrightarrow{\mathcal F}_{\mathfrak{sl}}^1,
\qquad
d^1_{{\mathcal F}^{\bullet}_{\mathfrak{sl}}}
:=d^1_{{\mathcal F}^{\bullet}}|_{{\mathcal F}^1_{\mathfrak{sl}}}\colon
{\mathcal F}_{\mathfrak{sl}}^1\longrightarrow{\mathcal F}_{\mathfrak{sl}}^2.
\]
Then the deformation obstruction for the smoothness of
$M^{\alpha}_{\mathrm{Hod}}\rightarrow N$ lies in 
$\mathbb{H}^2({\mathcal F}_{\mathfrak{sl}}^{\bullet})$.
So we only have to prove $\mathbb{H}^2({\mathcal F}_{\mathfrak{sl}}^{\bullet})=0$.
There is a canonical map
${\mathcal F}_{\mathfrak{sl}}^{\bullet}\rightarrow{\mathcal F}^{\bullet}$
and ${\mathcal F}^{\bullet}$
is quasi-isomorphic to 
${\mathcal F}_{\mathfrak{sl}}^{\bullet}\oplus\Omega_X^{\bullet}$.
So we have
\begin{equation} \label {trace decomposition of H^2}
\mathbb{H}^2({\mathcal F}^{\bullet})
=\mathbb{H}^2({\mathcal F}_{\mathfrak{sl}}^{\bullet})
\oplus \mathbb{H}^2(\Omega_X^{\bullet}).
\end{equation}
The exact sequence \eqref{long exact sequence for tangent}
and the commutative diagram \eqref{diagram for five lemma}
induce the commutative diagram
\[
\begin{CD}
\mathbb{H}^1({\mathcal F}_0^{\bullet})
@>>> \mathbb{H}^1({\mathcal F}_1^{\bullet})
@>>> \mathbb{H}^2({\mathcal F}^{\bullet})
@>>> 0
\\
@V ^t\tau_2 V \cong V  @V ^t\tau_1 V\cong V  @VVV
\\
\mathbb{H}^0({\mathcal F}_1^{\bullet})^{\vee}
@>>> \mathbb{H}^0({\mathcal F}_0^{\bullet})^{\vee}
@>>> \mathbb{H}^0({\mathcal F}^{\bullet})^{\vee}
@>>> 0 \rlap{,}
\end{CD}
\]
which yields the isomorphism
$\mathbb{H}^2({\mathcal F}^{\bullet})\xrightarrow{\sim}
\mathbb{H}^0({\mathcal F}^{\bullet})^{\vee}$.
An element $u$ of
$\mathbb{H}^0({\mathcal F}^{\bullet})$ is 
an endomorphism of $E$ preserving $\nabla$ and $\{l^x_i\}$.
Since $(E,\nabla,\{l^x_i,L^x_i\})$ is
$\alpha$-stable, $u$ must be a scalar multiplication.
Thus we have $\mathbb{H}^0({\mathcal F}^{\bullet})=\mathbb{C}$,
which yields $\mathbb{H}^2({\mathcal F}^{\bullet})\cong\mathbb{C}$.
Combining with 
$\mathbb{H}^2(\Omega_X^{\bullet})=H^2(X,\mathbb{C})\cong\mathbb{C}$,
\eqref {trace decomposition of H^2} yields
$\mathbb{H}^2({\mathcal F}_{\mathfrak{sl}}^{\bullet})=0$.
\end{proof}

\begin{Prop}
The fiber
$M^{\alpha}_{\mathrm{Hod}}(\lambda,\nu)$
of the moduli space $M^{\alpha}_{\mathrm{Hod}}$
over $(\lambda,\nu)\in \Lambda$
is of dimension $2r^2(g-1) +2 +r(r-1)\deg D$.
\end{Prop}

\begin{proof}
Since we have
$\mathbb{H}^0({\mathcal F}_1^{\bullet})
\cong\mathbb{H}^1({\mathcal F}_0^{\bullet})^{\vee}$
and
$\mathbb{H}^1({\mathcal F}_1^{\bullet})\cong
\mathbb{H}^0({\mathcal F}_0^{\bullet})^{\vee}$
by the proof of Lemma \ref{lemma: nondegenerate},
the long exact sequence \eqref{long exact sequence for tangent} yields
\begin{align*}
\mathbb{H}^1({\mathcal F}^{\bullet})
&=
\dim \mathbb{H}^0({\mathcal F}_1^{\bullet})
+\dim \mathbb{H}^1({\mathcal F}_0^{\bullet})
-\dim\mathbb{H}^1({\mathcal F}_1^{\bullet})
\\
&\qquad
-\dim\mathbb{H}^0({\mathcal F}_0^{\bullet})
+\dim\mathbb{H}^0({\mathcal F}^{\bullet})
+\dim\mathbb{H}^2({\mathcal F}^{\bullet})
\\
&=
2\dim\mathbb{H}^1({\mathcal F}_0^{\bullet})
-2\dim\mathbb{H}^0({\mathcal F}_0^{\bullet})+2
\\
&=2(1-\chi({\mathcal G}^0)-\length A^0+\length \overline{S}^2(E|_D)).
\end{align*}
By the Riemann--Roch formula, we have
\[
\chi({\mathcal G}^0)=
r^2(1-g)+\deg{\mathcal G}^0
=
-r^2(g-1)-\sum_{x\in D}\sum_{1\leq i<j\leq s_x}r^x_i r^x_jm_x.
\]
From the definition of $A^0$ and $\overline{S\,}^2(E|_D)$, we have
\[
\length A^0=\sum_{x\in D}\sum_{i=1}^{s_x}r^x_im_x \quad \text{and}
\quad
\length \overline{S\,}^2(E|_D)
=
\sum_{x\in D}\sum_{i=1}^{s_x}\frac{1}{2}r^x_i(r^x_i+1)m_x.
\]
Thus  we have
\begin{align*}
\dim \mathbb{H}^1({\mathcal F}^{\bullet})
&=
2+2r^2(g-1)+2\sum_{x\in D}\sum_{1\leq i<j\leq s_x}r^x_ir^x_jm_x
\\
&\qquad
-2\sum_{x\in D}\sum_{i=1}^{s_x}r^x_im_x
+\sum_{x\in D}\sum_{i=1}^{s_x}r^x_i(r^x_i+1)m_x
\\
&=2+2r^2(g-1)+\sum_{x\in D}\bigg(\Big(\sum_{i=1}^{s_x}r^x_i\Big)^2
-\sum_{i=1}^{s_x}r^x_i\bigg)m_x.
\end{align*}
Since $r=\sum_{i=1}^{s_x}r^x_i$ for any $x\in D$,
we get the desired equality.
\end{proof}

\subsection{Relative symplectic form on the moduli space}\label {subsect rel symp str}

Let 
$p\colon X\times M^{\alpha}_{\mathrm{Hod}}\rightarrow 
M^{\alpha}_{\mathrm{Hod}}$
be the projection to the second component.
There is a universal family
$(\tilde{E},\tilde{\nabla},\{\tilde{l}^x_i, {\mathcal A}^x_i \tilde{\sigma}^{x,i}   \}  )$
on $X\times M^{\alpha}_{\mathrm{Hod}}$,
after replacing $M^{\alpha}_{\mathrm{Hod}}$
with an \'etale covering if necessary.
We can define a complex $\tilde{\mathcal F}^{\bullet}$
on $X\times M^{\alpha}_{\mathrm{Hod}}$
in the same way as \eqref{complex defining tangent space}.
By Proposition \ref{Prop: tangent space},
$\mathbb{R}^1p_*(\tilde{\mathcal F}^{\bullet})$
is isomorphic to the relative tangent bundle
$T_{M^{\alpha}_{\mathrm{Hod}}/\Lambda}$.
The nondegenerate pairing \eqref{nondegenerate pairing}
can be extended to a pairing
\begin{equation} \label {symplectic form on the moduli space}
 \omega_{M^{\alpha}_{\mathrm{Hod}}/\Lambda} \colon 
 \mathbb{R}^1p_*(\tilde{\mathcal F}^{\bullet}) \times
 \mathbb{R}^1p_*(\tilde{\mathcal F}^{\bullet})
 \; \longrightarrow \; {\mathcal O}_{M^{\alpha}_{\mathrm{Hod}}}
\end{equation}
on the relative tangent bundle.
We can check by an easy calculation that
$\omega_{M^{\alpha}_{\mathrm{Hod}}/\Lambda}$
is a skew-symmetric pairing and define a relative $2$-form
on $M^{\alpha}_{\mathrm{Hod}}$ over $\Lambda$.
In order to complete the proof of Theorem \ref{Thm: main theorem},
we need the following proposition.

\begin{Prop} \label {Prop: d-closed}
The relative $2$-form 
$\omega_{M^{\alpha}_{\mathrm{Hod}}/\Lambda} 
\in H^0(M^{\alpha}_{\mathrm{Hod}},
\Omega^2_{M^{\alpha}_{\mathrm{Hod}}/\Lambda} )$
defined in \eqref{symplectic form on the moduli space}
is $d$-closed.
\end{Prop}

\begin{proof}
It is sufficient to prove
$d\omega_{M^{\alpha}_{\mathrm{Hod}}/\Lambda} 
\big|_{M^{\alpha}_{\mathrm{Hod}}(\lambda,\nu)}=0$
for generic $(\lambda,\nu)\in \Lambda$,
because $M^{\alpha}_{\mathrm{Hod}}$ is smooth over $\Lambda$.
We may assume $\lambda=1$ by replacing 
$\nabla$ with $\frac{1}{\lambda}\nabla$.
By choosing a sufficiently generic $\nu$,
we may assume that every connection
with a parabolic frame is irreducible.
Take generic $\gamma^x_1,\ldots,\gamma^x_{s_x}\in\mathbb{C}$
and set
${\mathcal A}^x_{i,t}:={\mathcal O}_{m_xx}[w]/(w^{r^x_i}-\gamma^x_it^{r^x_i}-z)$
for $t\in\mathbb{A}^1$.
For a Zariski open neighborhood $T$
of $t=0$ in $\mathbb{A}^1$,
consider the family of moduli spaces
\[
 {\mathcal M}\longrightarrow T
\]
whose fiber ${\mathcal M}_t$ over $t\in T$
is the moduli space of tuples $(E,\nabla,\{l^x_i\},\{[\sigma^{x,i}_t]\})$
consisting of
\begin{itemize}
\item[(i)] an algebraic vector bundle $E$ on $X$,
\item[(ii)] a connection 
$\nabla\colon E\rightarrow E\otimes\Omega^1_X(D)$,
\item[(iii)]
a filtration $E|_{m_xx}=l^x_1\supset l^x_2\supset\cdots\supset l^x_{s_x}=0$
such that $l^x_i/l^x_{i+1}\cong{\mathcal O}_{m_xx}^{\oplus r^x_i}$
for each $x\in D$,
\item[(iv)]
an isomorphism
$\sigma^{x,i}_t\colon {\mathcal A}^x_{i,t}\xrightarrow{\sim} l^x_i/l^x_{i+1}$
such that
$\sigma^{x,i}_t\circ\nu(w)\mathrm{id}_{{\mathcal A}^x_{i,t}}
=\overline{\nabla}^x_i\circ \sigma^{x,i}_t$
and $[\sigma^{x,i}_t]$ is the equivalence class of $\sigma^{x,i}_t$
with respect to the action of $({\mathcal A}^x_{i,t})^{\times}$.
\end{itemize}
By shrinking $T$,
we may assume that every $(E,\nabla,\{l^x_i\},\{ [\sigma^{x,i}_t]\})$ belonging to
${\mathcal M}$ is irreducible.
Then ${\mathcal M}$ becomes smooth over $T$
by the same proof as in Proposition \ref{Prop: smoothness of moduli}.

There is a universal family 
\[
(\tilde{E}_{\mathcal M},\tilde{\nabla}_{\mathcal M},
\{(\tilde{l}_{\mathcal M})^x_i\}, \{ [\tilde{\sigma}^{x,i}_{\mathcal M}]\})
\]
on $X\times {\mathcal M}$
and we can define a complex ${\mathcal F}_{\mathcal M}^{\bullet}$
from $(\tilde{E}_{\mathcal M},\tilde{\nabla}_{\mathcal M},
\{(\tilde{l}_{\mathcal M})^x_i\}, \{ [\tilde{\sigma}^{x,i}_{\mathcal M}]\})$
in the same way as \eqref{complex defining tangent space}.
We can also define a relative $2$-form
\[
\omega_{\mathcal M}\colon
\mathbb{R}^1p_*({\mathcal F}_{\mathcal M}^{\bullet})
\times \mathbb{R}^1p_*({\mathcal F}_{\mathcal M}^{\bullet})
\longrightarrow {\mathcal O}_{\mathcal M}
\]
on ${\mathcal M}$ over $T$
in the same way as \eqref{symplectic form on the moduli space},
or \eqref{nondegenerate pairing}.
When $t\neq 0$, there exist distinct complex numbers
$\mu_1,\ldots,\mu_{r^x_i}$ such that
\begin{align*}
{\mathcal A}^x_{i,t}
=
{\mathcal O}_{m_xx}[w]/(w^{r^x_i}-\gamma^x_it^{r^x_i}-z)
&\cong
{\mathcal O}_{m_xx}\times\cdots\times{\mathcal O}_{m_xx}
\\
&\cong
{\mathcal O}_{m_xx}[u]/(u-\mu_1)\cdots(u-\mu_{r^x_i}).
\end{align*}
Therefore the fiber ${\mathcal M}_t$ over $t \neq 0$
can be identified with the moduli space of stable generic unramified
irregular singular connections constructed in \cite[Theorem 2.11]{Inaba2},
and the restriction
${\mathcal F}_{\mathcal M}^{\bullet}|_{X\times \{t\}}$
is essentially the same as the complex ${\mathcal F}^{\bullet}$
defined in \cite[Section 2.4]{Inaba2}.
Moreover, the restriction $\omega_{\mathcal M}|_{{\mathcal M}_t}$
coincides with the $2$-form given in \cite[Proposition 2.26]{Inaba2},
which yields $d\omega_{\mathcal M}|_{{\mathcal M}_t}=0$.
The smoothness of ${\mathcal M}$ over $T$
yields $d\omega_{\mathcal M}=0$, which implies
$d\omega_{M^{\alpha}_{\mathrm{Hod}}/\Lambda} 
\big|_{M^{\alpha}_{\mathrm{Hod}}(\lambda,\nu)}=
d\omega_{{\mathcal M}}|_{{\mathcal M}_0}=0$.
\end{proof}

Note that the relative symplectic form
$\omega$ in \eqref{symplectic form on the moduli space}
induces an isomorphism
\[
 b\colon
 T_{M^{\alpha}_{\mathrm{Hod}}/\Lambda}
 \; \xrightarrow{\;\sim\;} \;
 \Omega^1_{M^{\alpha}_{\mathrm{Hod}}/\Lambda}.
\]
For $f\in{\mathcal O}_{M^{\alpha}_{\mathrm{Hod}}}$,
define a vector field $X_f\in T_{M^{\alpha}_{\mathrm{Hod}}/\Lambda}$
by
$X_f:=b^{-1}(-df)$,
where $df\in\Omega^1_{M^{\alpha}_{\mathrm{Hod}}/\Lambda}$
is the relative differential of $f$.
In other words,
$ \omega(X_f,v)=-df(v)$
for any $v\in T_{M^{\alpha}_{\mathrm{Hod}}/\Lambda}$.
Define a pairing 
$\{-,-\}\colon {\mathcal O}_{M^{\alpha}_{\mathrm{Hod}}}\times
{\mathcal O}_{M^{\alpha}_{\mathrm{Hod}}}\rightarrow
{\mathcal O}_{M^{\alpha}_{\mathrm{Hod}}}$ by
\begin{equation} \label {Poisson bracket}
\{ f,g\}:=
\omega(X_f,X_g)
\end{equation}
for $f,g\in{\mathcal O}_{M^{\alpha}_{\mathrm{Hod}}}$.
Then we have the following:

\begin{Cor}  \label {Cor: Poisson structure}
The pairing
$\{-,-\}\colon {\mathcal O}_{M^{\alpha}_{\mathrm{Hod}}}\times
{\mathcal O}_{M^{\alpha}_{\mathrm{Hod}}}\rightarrow
{\mathcal O}_{M^{\alpha}_{\mathrm{Hod}}}$
defined by \eqref{Poisson bracket}
becomes a Poisson bracket.
\end{Cor}

\begin{proof}
The axioms of a Poisson bracket other than the Jacobi identity
are straightforward. 
The Jacobi identity follows from the relative
$d$-closedness of $\omega$.
\end{proof}

\end{document}